\documentclass[11pt]{amsart}

\usepackage{amssymb,amsfonts}
\usepackage{amsmath,amsthm,amscd}
\usepackage{mathrsfs} 
\usepackage{enumerate}
\usepackage{verbatim}
\usepackage{color}
\usepackage[colorlinks]{hyperref}
\usepackage[displaymath, tightpage]{preview}
\usepackage{beton}

\def\dist{\operatorname{dist}}
\def\fin{\text{fin}}
\def\osc{\operatorname{osc}}
\def\A{\operatorname{\mathscr A}}

\newtheorem{thm}{Theorem}[section]
\newtheorem{defi}[thm]{Definition}
\newtheorem{lem}[thm]{Lemma}
\newtheorem{cor}[thm]{Corollary}
\newtheorem{prop}[thm]{Proposition}
\newtheorem{fact}[thm]{Fact}

\theoremstyle{remark}
\newtheorem{remark}[thm]{Remark}

\newtheorem{example}[thm]{Example}

\begin{document}

\title{Splitting chains, tunnels and twisted sums}

\author[Cabello-S\'anchez]{F\'{e}lix Cabello S\'anchez}
\address{Departamento de Matem\'aticas and IMUEx, Universidad de Extremadura, 06071 Badajoz (Spain)} 
\email{fcabello@unex.es}

\author[Avil\'es]{Antonio Avil\'{e}s}
\address{Departamento de Matem\'aticas, Universidad de Murcia, 30100 Murcia (Spain)} 
\email{avileslo@um.es}

\author[Borodulin-Nadzieja]{Piotr Borodulin-Nadzieja}
\address{Instytut Matematyczny, Uniwersytet Wroc\l awski, Wroc\l aw (Poland)} 
\email{pborod@math.uni.wroc.pl}

\author[Chodounsk\'{y}]{\\David Chodounsk\'{y}}
\address{Institute of Mathematics of the Czech Academy of Sciences, Praha 1 (Czech Republic)} 
\email{david.chodounsky@matfyz.cz}

\author[Guzm\'{a}n]{Osvaldo Guzm\'{a}n}
\address{University of Toronto, Toronto (Canada)}
\email{oguzman@math.utoronto.ca}

\thanks{AA was supported by projects MTM2017-86182-P (AEI, Government of Spain and ERDF, EU) and 20797/PI/18 by Fundaci\'{o}n S\'{e}neca, ACyT Regi\'{o}n de Murcia. PBN was supported by Polish National Science Center grant no 2018/29/B/ST1/00223. He is also
	indebted to Universidad de Murcia for financing his stay in Murcia, during which AA introduced him to the concept of splitting chains.
FCS was supported in part by DGICYT project MTM2016-76958-C2-1-P (Spain) and
Junta de Extremadura program IB-16056.
DCh was supported by the GACR project 17-33849L and RVO:
67985840. OG was supported by NSERC grant number 455916.}

\date{July 22, 2019}

\subjclass[2010]{03E17,03E75,46B25,46E15,54A35}
\keywords{twisted sums of Banach spaces, short exact sequences of Banach spaces, splitting families, tight gaps, complete tunnels, Aronszajn trees, Cohen forcing}

\begin{abstract}  
	We study splitting chains in $\mathscr{P}(\omega)$, that is, families of subsets of $\omega$ which are linearly ordered by $\subseteq^*$ and which are splitting. We prove that their existence is independent of axioms of $\mathsf{ZFC}$. We show that
they can be used to construct certain peculiar Banach spaces: twisted sums of $C(\omega^*)=\ell_\infty/c_0$ and $c_0(\mathfrak c)$. Also, we consider splitting chains in a topological setting, where they give rise to the so called tunnels.  \end{abstract}

\maketitle

\section{Introduction}
We say that a compact space $K$ has a tunnel, if there is a continuous mapping $f\colon K\longrightarrow L$, where $L$ is a linearly ordered topological space, such that $f^{-1}(x)$ is nowhere dense in $K$ for each $x\in L$. This notion was introduced by
Nyikos in \cite{Nyikos} (under the name of \emph{a complete tunnel}). Although it may seem like the spaces with tunnels should resemble in a sense linearly ordered topological spaces, in fact the property of possessing a tunnel is quite widespread among compact spaces without isolated points.
Actually, it is not easy to find a compact space without isolated points which does not have a tunnel. 

In this article we are going to study the notion of tunnel and some of its variations in the context of infinitary combinatorics, topology and homological Banach space theory. 

We will be mostly interested in the question if $\omega^*$ (that is $\beta\omega\setminus \omega$, the remainder of the \v Cech-Stone compactification of $\omega$) has a tunnel.  This question is interesting in all the settings mentioned above. In
particular, it is connected to the existence of certain peculiar family of subsets of $\omega$. A family $\mathscr{S}$ of subsets of the set $\omega$ of natural numbers is called splitting if for every infinite set $A\subseteq \omega$ there exists $S\in \mathscr{S}$ such that both $A\cap S$ and $A\setminus S$ are infinite. Splitting families are well studied
objects in set theory, specially in connection with the important cardinal $\mathfrak{s}$, the least cardinality of such a family. In this paper we are interested in splitting families which are moreover chains in the almost inclusion order, that is,
with the extra property that if $A,B\in\mathscr{S}$, then either $A\setminus B$ or $B\setminus A$ is finite.
Nyikos proved several results about existence of splitting chains in various models of set theory: he showed e.g. that $\mathsf{PFA}$ implies that there are no splitting chains whereas. In this
article we partially follow Nyikos' path (although most of the results we proved before we have discovered Nyikos' work), reproving some of his theorems and
proving that splitting chains do exist in the standard Cohen model (Nyikos announced that the proof of that result would appear in a later paper that, however, was never published). Also, we show that the existence of a splitting chain is compatible with the assumption that $\mathfrak{p}>\omega_1$.


Our initial motivation for the study of tunnels and splitting chains
stems from their uses in the construction and analysis of certain ``twisted sums'' of Banach spaces.
Let us recall that a short exact sequence of Banach spaces is a  diagram of Banach spaces and (linear, continuous) operators
\begin{equation}\label{sex}
\begin{CD}
0@>>> Y @>\imath>> Z @>\pi>> X@>>> 0
\end{CD}
\end{equation}
in which the kernel of each arrow agrees with the range of the preceding one.
The middle space $Z$ is often called a ``twisted sum'' of $Y$ and $X$ ... in that order!
One says that such an exact sequence is trivial
(or that it splits, but we prefer to avoid this terminology in this paper) if the mapping $\imath$ admits an inverse (see Section \ref{derivation2} for the precise definition) in which case the twisted sum space $Z$ is ``well isomorphic'' to the direct sum $Y\oplus X$.


Questions about whether a Banach space $Z$ contains a copy of some classical Banach space $Y$ (or it does not) are central in Banach space theory (recall, e.g., the celebrated Rosenthal's $\ell_1$ theorem \cite{Rosenthal74}, or Bessaga-Pe\l czy\'nski theorem characterizing Banach spaces containing a copy of $c_0$, \cite{Bessaga58}).

We may ask about how a subspace $Y$ is situated in $Z$ considering also the quotient space $X=Z/Y$ and the exact sequence
\begin{equation}\label{yzz/y}
\begin{CD}
0@>>> Y @>\imath>> Z @>\pi>> Z/Y@>>> 0
\end{CD}
\end{equation}
where $\imath$ is the inclusion and $\pi$ is the natural quotient map.
In this context, $Y$ is complemented in $Z$ if and only if (\ref{yzz/y}) is trivial; if the sequence is not trivial, then $Y$ lies in $Z$ in a non-trivial way and $Z$ is kind of an
\emph{alloy} of $Y$ and $X$, but not a \emph{trivial} one. Determining for which Banach spaces $Y$ and $X$ one can construct a nontrivial sequence (\ref{sex}) is a fundamental question in the homological theory of Banach spaces (see the monograph \cite{castillo} for a general account; the approaches of the more recent papers \cite{avilesAiM}, \cite{Felix03}, \cite{Plebanek18} are more akin to ours). 

As we shall see, each tunnel of $K$ induces an exact sequence of the form
\begin{equation*}\begin{CD}
0@>>> C(K) @>>> Z @>>> c_0(\kappa)@>>> 0
\end{CD} 
\end{equation*} which is nontrivial if the tunnel has some additional properties, for instance, when it is made of regular open sets. Here, and throughout the paper, $C(K)$ denotes the Banach space of all continuous functions $f\colon
K\longrightarrow\mathbb R$ with the sup norm. Also, if $I$ is a set, then $c_0(I)$ denotes the space of all functions $f\colon I\longrightarrow\mathbb R$ such that, for every $\varepsilon>0$, the set $\{i\in I\colon |f(i)|> 0\}$ is finite, again with the sup norm.

The most interesting case is, by far, when $K = \omega^\ast$ is the \v{C}ech-Stone remainder of the natural numbers. Our main result in this line is that if there is a splitting chain of clopens in $\omega^\ast$, then  a nontrivial exact sequence
\begin{equation}\label{c01}
\begin{CD}
0@>>> C(\omega^\ast) @>>> Z @>>> c_0(\mathfrak{c}) @>>> 0.
\end{CD}
\end{equation}		 
exists; see Theorem \ref{the-main}.
Twisted sums of the form
\begin{equation}\label{cwstarY}\begin{CD}
0@>>> C(\omega^\ast) @>>> Z @>>> X @>>> 0
\end{CD}\end{equation}
were constructed first by Amir~\cite{amir} and later by Avil\'{e}s and Todorcevic~\cite{multigaps}, but not much is specified in these constructions about the structure of $X$. Amir's construction is described in \cite[Section 2.5, Proposition 2.43]{aviles}. More recently, the authors of \cite{corr} constructed a twisted sum like
(\ref{c01}) under $\mathsf{CH}$. As we have mentioned above, we are able to prove the existence of splitting chains of clopens in $\omega^\ast$ under several other assumptions, and so by Theorem \ref{the-main} we get new examples of twisted sums of the form
(\ref{c01}).


The paper is organized as follows:
In Section~\ref{sectiongeneral} we study the notion of tunnels and splitting chains in the more general context of topological spaces. Section~\ref{derivation2} relates tunnels and splitting chains with twisted sums of Banach spaces, and finally in Section~\ref{forcing} we prove the consistency results concerning the existence of splitting chains in $\mathscr{P}(\omega)/\fin$.


\section{Tunnels and splitting chains of open sets} \label{sectiongeneral}

We assume that all topological spaces are Hausdorff. Let $U$, $V$ be open in a topological space $X$. By $U < V$ we will denote the relation $\overline{U} \subseteq V$. 
The relation ``$\leq$'' defined by \[ U\leq V \mbox{     if     } (U<V \mbox{ or } U=V)\] is a partial order on the topology of $X$.

\begin{defi} We say that a family of open subsets of $X$ is a \emph{chain (of open sets)} if it is linearly ordered by $\leq$.
\end{defi}

\begin{defi} A chain of open subsets $\mathscr{U}$ of $X$ is a \emph{tunnel} if the set $\bigcup \{\partial U\colon U\in \mathscr{U}\}$ is dense in $X$.
\end{defi}

Clearly, no space with an isolated point can have a tunnel. However, the property of having a tunnel is quite common among spaces without isolated points. We will begin with some easy examples.

\begin{example} The family of all open balls with fixed center forms a tunnel in Euclidean spaces. More generally if $(X,d)$ is a metric space with a point $x_0\in X$ such that for each $r>0$
	\[ \overline{B(x_0,r)} = \{x\in X\colon d(x,x_0)\leq r\} \]
	then the family 
	$\{B(x_0,r)\colon r>0\}$ is a tunnel. Consequently, normed spaces have tunnels. 	\medskip
\end{example}

Less obvious examples of spaces with tunnels are given by the following proposition.

\begin{prop}\label{sep}
	If $X$ is a separable normal space without isolated points, then $X$ has a countable tunnel.
\end{prop}

\begin{proof} 
First, notice that if $V<U$ are open subsets of $X$ and $x\in U\setminus \overline{V}$, then there is an open $W$ such that $V<W<U$ and $x\in \partial W$. Indeed, use normality to find an open neighbourhood $W_0$ of $x$ such that $W_0 < U\setminus \overline{V}$. Since $x$ is not an isolated point, $x\in \partial(W_0 \setminus \{x\})$. Now, again by normality, there is $W_1$ such that $V<W_1<U$
and $x\notin W_1$. The set $W = W_0\setminus\{x\} \cup W_1$ is as desired.

Let $D=\{d_n\colon n\in\omega\}$ be a dense subset of $X$. Using the above remark it is easy to construct inductively a chain of open sets $\{U_n\colon n\in\omega\}$ such that $d_n \in \partial U_n$. Of course $\{U_n\colon n\in\omega\}$ is a
tunnel.
\end{proof}

\begin{cor}\label{cor:metriz-have}
	Every compact metrizable space without isolated points has a countable tunnel.
\end{cor}

\begin{remark} A study of tunnels in metric spaces was undertaken by Maciej Niewczas in \cite{Niewczas}. Witold Marciszewski proved that the assumption on compactness is obsolete and in fact every metrizable space without isolated points has a countable tunnel, using the fact that every metric space has a $\sigma$-discrete base. 
\end{remark}

\begin{prop} If $X$ has a tunnel and $Y$ is a topological space, then $X\times Y$ has a tunnel.
\end{prop}

\begin{proof}
	Let $\mathscr{U}$ be a tunnel in $X$. It is easy to verify that $\{U\times Y\colon U\in \mathscr{U}\}$ is a tunnel in $X\times Y$.
\end{proof}
\begin{cor} The space $2^\kappa$ has a countable tunnel for each infinite $\kappa$.
\end{cor}

\begin{proof} $2^\kappa = 2^\omega \times 2^{\kappa \setminus \omega}$ and $2^\omega$ has a tunnel, according to Proposition \ref{sep}.
\end{proof}

We say that $A$ \emph{splits} $B$ if both $B\cap A$ and $B\setminus A$ are nonempty. A family $\mathscr{A}$ of subsets of $X$ is \emph{splitting} if every nonempty open subset of $X$ is split by some member of $\mathscr{A}$.
It will be convenient to notice that splitting families which form chains of open sets satisfy a slightly stronger splitting property.

\begin{lem}\label{spli} If $\mathscr{U}$ is a splitting chain, then for each nonempty open $V\subseteq X$ there is $U\in \mathscr{U}$ such that $V\cap U \ne \varnothing$ and $V\setminus \overline{U} \ne \varnothing$.
\end{lem}
\begin{proof} Let $V\subseteq X$ be a nonempty open set. Consider $U' \in \mathscr{U}$ which splits $V$ and then $U\in \mathscr{U}$ which splits $U'\cap V$. We have $U<U'$ because $U'\leq U$ is impossible. Clearly, $V \cap U \ne \varnothing$ and $V\setminus \overline{U} \supseteq V\setminus U' \ne \varnothing$.
\end{proof}

The proof of the following simple fact is left to the reader.

\begin{prop}\label{tunnel-chain}
Every tunnel is a splitting chain.
\end{prop}

Of course not every splitting chain is a tunnel, but in compact spaces each splitting chain can be used to produce a tunnel.

\begin{prop} \label{chain-tunnel}Assume that $\mathscr{U}$ is a splitting chain in a compact space $K$. Then 
	\[ \mathscr{V} = \left\{\bigcup \mathscr{U}'\colon \varnothing \ne \mathscr{U}'\subseteq \mathscr{U} \mbox{ and } \mathscr{U}'\mbox{ does not have a }\leq\mbox{-maximal element} \right\} \] 
	forms a tunnel in $K$. Moreover, $\mathscr{V}$ has the following properties:
	\begin{enumerate}
		\item \label{p1} if $\mathscr{V}'\subseteq \mathscr{V}$, then $\bigcup \mathscr{V}' \in \mathscr{V}$;
		\item \label{p2} if $V<U\in \mathscr{V}$, then there is $W\in \mathscr{V}$ such that $V<W<U$.
	\end{enumerate}
\end{prop}
\begin{proof} First, we will show that $\mathscr{V}$ is a chain. Let $V_0$, $V_1$ be distinct elements of $\mathscr{V}$ and let $V_0 = \bigcup \mathscr{U}_0$, $V_1 = \bigcup \mathscr{U}_1$, where $\mathscr{U}_0$ and $\mathscr{U}_1$ are subfamilies of
	$\mathscr{U}$ without maximal elements. Without loss of generality, we may assume that there is
	$U_1 \in \mathscr{U}_1$ such that $U<U_1$ for each $U\in \mathscr{U}_0$ and so $V_0 \subseteq U_1$. Since $U_1$ is not maximal in $\mathscr{U}_1$, there is $U_1 < U_2 \in \mathscr{U}_1$. Hence $V_0 \subseteq U_1 < U_2 \subseteq V_1$ and so $V_0 <
	V_1$. So, $\mathscr{V}$ is a chain.

	Now let $W$ be a nonempty open subset of $K$ and let $W'$ be a nonempty open set such that $W' < W$ (since $K$ admits a splitting chain, it cannot have an isolated point, so there is such $W'$). Using Lemma \ref{spli} we can recursively find a sequence $(U_n)$ of elements
	of $\mathscr{U}$ such that $U_n < U_{n+1}$ and $U_n$ splits $W'$ for every $n$. Then  $V = \bigcup_n U_n \in \mathscr{V}$. By compactness, there is \[ x\in \overline{V}\cap \overline{W'} \setminus V. \]
	But this means that $x\in \partial V \cap \overline{W'}$ and so $x \in \partial V \cap W$.

	It is straightforward to check that $\mathscr{V}$ has properties (\ref{p1}) and (\ref{p2}). 
\end{proof}

A variant of the above proposition in which we take only countable unions will be important for us later:

\begin{prop}\label{forantonio}
	Assume that $\mathscr{U}$ is a splitting chain in a compact space $K$. Then 
	\[ \mathscr{V}_\omega = \left\{\bigcup_{n\in \omega} W_n : W_n\in\mathscr{U}, W_1 < W_2<\cdots \right\} \] 
	forms a tunnel in $K$.
\end{prop}

\begin{proof}
	Just the same proof as the previous proposition.
\end{proof}

\begin{remark}\label{clopen} Thanks to Proposition \ref{chain-tunnel} to show that a compact zerodimensional space has a tunnel it is enough to find a family of clopens $\mathscr{C}$ which is linearly ordered by $\subseteq$ and which is splitting (although, the tunnel 
	produced according to the recipe from Proposition \ref{chain-tunnel} does not contain any element of $\mathscr{C}$). It is however unclear for us if the existence of a tunnel in a compact zerodimensional space implies the existence of
	such chain of clopens.
\end{remark}

\begin{thm}\label{Equiv} Let $K$ be a compact space. The following are equivalent:
	\begin{itemize}
		\item[(a)] $K$ has a tunnel.
		\item[(b)] $K$ has a splitting chain of open sets.
		\item[(c)] There is a continuous mapping $f\colon K \longrightarrow L$, where $L$ is a linearly ordered space, whose fibers are nowhere dense (i.e. $f^{-1}(l)$ is nowhere dense for each $l\in L$).
	\end{itemize}
\end{thm}

\begin{proof}
	The equivalence of (a) and (b) follows from Proposition \ref{chain-tunnel} and Proposition \ref{tunnel-chain}.
	\medskip

	(b)$\implies$(c)\,  Assume that $\mathscr{V}$ is a tunnel in $K$. We may immediately assume that it has properties (\ref{p1}) and (\ref{p2}) of Proposition \ref{chain-tunnel}. Now, equip $\mathscr{V}$ with the order
	topology with respect to ``$\leq$''. Define $f\colon K \longrightarrow \mathscr{V}$ by
	\[ f(x) = \bigcup \{U\in \mathscr{V}\colon x\notin U\}. \]
	Assume that  $V< U\in \mathscr{V}$ and $f(x_0)\in (V,U)$ and notice that $x_0\in U$. In order to verify the continuity of $f$ we will show that $x_0$ has an open neighbourhood contained in $f^{-1}(V,U)$. 
	First, notice that for each $x\in K$ we have $x\notin f(x)$. Therefore, if $x\in U$, then $f(x)<U$. Second, if $V<W$, $W\in \mathscr{V}$ and $x\notin W$, then $f(x)>V$. Now, let $W\in \mathscr{V}$ be such that $V<W<f(x_0)$ and notice that
	$x_0\notin \overline{W}$. Then, using the above remarks, we have \[ x_0 \in U\setminus \overline{W}\subseteq
	f^{-1}(V,U).\] 

\medskip

(c)$\implies$(b)\, Suppose $f\colon K \longrightarrow L$ is a mapping with the desired properties. For $l\in L$ by $(-\infty,l)$ we will denote the set $\{x\in L\colon x < l\}$, where $\leq$ is the linear ordering compatible with the topology of
$L$. We will understand $(-\infty,l]$ in the similar way. For $l\in L$ let 
\[ U_l = f^{-1}(-\infty,l). \]
We claim that $\{U_l\colon l\in L\}$ is a splitting chain of open sets.
First, we will check that it is a chain. Let $l < l' \in L$. Then, by continuity of $f$,
\[ \overline{U_l} = \overline{f^{-1}(-\infty,l)} \subseteq f^{-1}\overline{(-\infty,l)} \subseteq f^{-1}(-\infty,l] \subseteq f^{-1}(-\infty, l') = U_{l'}. \]
To show that $\{U_l\colon l\in L\}$ is splitting, consider a nonempty open set $V\subseteq K$ and notice that, by the assumption on $f$, we can find $l_0<l_1$ in $f[V]$. Then $U_{l_1}$ splits $V$.
\end{proof}

\begin{remark} In \cite{Nyikos} Nyikos introduced the notion of \emph{complete tunnel}. A chain of open subsets $\mathscr{U}$ of $X$ is a complete tunnel if for every $\mathscr{U}'\subseteq \mathscr{U}$ we have \[ \mathrm{Int}\Big{(}\bigcap\{U\in
	\mathscr{U}\colon U' < U \mbox{ for every }U'\in \mathscr{U}'\}\Big{)} \subseteq \overline{\bigcup\mathscr{U}'}. \]
	Nyikos proved (\cite[Theorem 1.5]{Nyikos}) that being a complete tunnel is equivalent to (c) of Theorem~\ref{Equiv} and so it is equivalent, at least in the realm of compact spaces, to being a tunnel in our sense.
\end{remark}

\begin{cor}\label{Countable-split} Let $K$ be a compact space without an isolated point and let $L$ be linearly ordered and metrizable. Assume that there is a continuous mapping  $f\colon K \longrightarrow L$ with nowhere dense fibers. Then $K$ has a countable splitting chain of $F_\sigma$-open sets.
\end{cor}

\begin{proof}
	Let $D$ be a countable dense subset of $L$ and let $\mathscr{U} = \{U_d\colon d\in D\}$, where $U_x = f^{-1}(-\infty, x)$. That $\mathscr{U}$ is splitting can be proved in the same way as in proof of Theorem \ref{Equiv}, (c) $\implies$ (b).
\end{proof}

Recall that an \emph{interval algebra} is a Boolean algebra generated by a chain. The Stone space of an interval algebra is linearly ordered (and the Boolean algebra of clopens of a linearly ordered zerodimensional compact space forms an interval algebra). 

\begin{cor}\label{interval} If $\mathfrak{A}$ is a Boolean algebra which contains an interval subalgebra $\mathfrak{B}$ which splits nonempty elements of $\mathfrak{A}$, then the Stone space of $\mathfrak{A}$ has a tunnel. 
\end{cor}

\begin{proof}
	Let $K$, $L$ be the Stone spaces of $\mathfrak{A}$, $\mathfrak{B}$ respectively. Then there is a canonical continuous surjection $f\colon K \longrightarrow L$, where $L$ is a linearly ordered space. If $V$ is a clopen subset of $K$, then it is split by some $B\in
	\mathfrak{B}$. So, if $x\in V\cap B$ and $y\in V\setminus B$, then $f(x)\ne f(y)$ and so $f$ has nowhere dense fibers.
\end{proof}

Recall that a measure $\mu$ on a compact space $K$ is \emph{strictly positive} if $\mu(U)>0$ for each nonempty open set $U\subseteq K$. A measure $\mu$ is \emph{atomless} if $\mu(\{x\})=0$ for every $x\in K$. If $K$ is zerodimensional, then $\mu$ is
atomless if and only if for every $\varepsilon>0$ there is a partition of $K$ into clopen sets of measure at most $\varepsilon$.

\begin{prop} Every compact zerodimensional space supporting a strictly positive atomless probability measure has a tunnel.
\end{prop}

\begin{proof} Assume $K$ supports such a measure $\mu$.
	It is enough to construct a chain $\mathscr{C}$ of clopen subsets of $K$ such that $\{\mu(C)\colon C\in \mathscr{C}\}$ is dense in $[0,1]$. Indeed, suppose that $\mathscr{C}$ has this property and $V$ is a nonempty clopen subset of $K$. Then
	$\mu(V)=r>0$. Let
	$R=\sup\{\mu(C)\colon C\in \mathscr{C}\mbox{ and } C\cap V=\varnothing\}\leq 1-r$ and consider $C\in \mathscr{C}$ such that $\mu(C)\in (R,R+r)$. Then $C\cap V\ne \varnothing$ because $\mu(C)>R$. If $V\subseteq C$, then $C$ contains $V$ and all the $C'\in\mathscr{C}$ such that $C'\cap V =\varnothing$, so we would have $\mu(C)>R+r$. We conclude that $C$ splits $V$. Now we can use Theorem \ref{Equiv}.

	One can construct $\mathscr{C}$ inductively subsequently using the following remark. Assume $C$ is non-empty clopen subsets of $K$. Then, by non-atomicity of $\mu$, there is a clopen set $D$ such that $D\subseteq C$ and
	$\mu(D)\in (\mu(C)/4, 3\mu(C)/4)$.	
\end{proof}

\subsection{Spaces without tunnels}
After so many examples of spaces having tunnels, we have to face the natural question: are there compact spaces without isolated points and without tunnels?

Recall that a compact space $K$ is \emph{Corson compact} if it can be embedded into \[ \Sigma(\mathbb{R}^\alpha) = \{x\in \mathbb{R}^\alpha\colon \{\xi\in \alpha\colon x(\xi) \ne 0\} \mbox{ is countable} \} \]
for some $\alpha$.

\begin{lem}\label{Corson} If $K$ is Corson compact and $K$ has a splitting chain, then $K$ has a countable splitting
	chain of $F_\sigma$ open sets.
\end{lem}
\begin{proof}
	According to Theorem~\ref{Equiv}, we have a continuous mapping $f:K\longrightarrow L$ with nowhere dense fibers onto a linearly ordered compact space. By \cite[IV.3.15]{Arhan}, $L$ is a Corson compact space. But every linearly ordered Corson compact space is metrizable \cite{eficer}, so we can use Corollary \ref{Countable-split}. 
\end{proof}

Recall than an \emph{Aronszajn tree} is an uncountable tree without an uncountable level and without an uncountable branch. Notice that Aronszajn trees are of height $\omega_1$. We will say that a Boolean algebra $\mathfrak{A}$ is \emph{Aronszajn} if it is
generated by an Aronszajn tree $T$, in the sense that there is a set of generators $\{a_t : t\in T\}$ of $\mathfrak{A}$ such that $a_t \leq a_s$ when $s\leq t$ and $a_t\cap a_s = 0$ when $s$ and $t$ are incomparable.

	\begin{thm}\label{aronszajn} Stone spaces of Aronszajn algebras do not have tunnels.
\end{thm}
\begin{proof}
	Let $T$ be an Aronszajn tree and let $\mathfrak{A}$ be the Boolean algebra generated by $T$. Let $K$ be the Stone space of $\mathfrak{A}$. Notice that  $K$ is Corson compact. Indeed, let $g\colon K \longrightarrow 2^T$ be given by $g(x)(t) = 1$ if and only if $t\in x$. It is plain to check that $g$ is a continuous embedding. Moreover, there is no $y\in f[K]$ of an uncountable support (since
	then $T$ would contain an uncountable branch).  If $K$ has a tunnel, then according to Lemma \ref{Corson}, $K$ has a countable splitting chain $\mathscr{U}$ of $F_\sigma$ open sets. Since each open set in $\mathscr{U}$ is $F_\sigma$, it is a countable union of clopen sets, and there is a countable ordinal $\alpha<\omega_1$ such that each element of $\mathscr{U}$ is in the algebra generated by the elements $a_t$ with $t$ of height less than $\alpha$. Consider now $s\in T$ of height greater than $\alpha$. Then $a_s$ is not split by any element of height less than $\alpha$ (each $a_t$ of height less than $\alpha$ either contains or is disjoint from $a_s$). Hence $a_s$ is not split by any element of $\mathscr{U}$, a contradiction.
\end{proof}


We finish with a remark which indicates that seeking for a compact space without tunnels (and isolated points) we should rather focus on spaces with many disjoint open sets. Recall that a topological space is \emph{ccc} if it does not contain an uncountable
family of nonempty open subsets.

\begin{prop} Suslin Hypothesis is equivalent to the assertion that every ccc compact zerodimensional space without isolated points has a tunnel.
\end{prop}

\begin{proof} ($\implies$) Suppose that $K$ is a ccc compact zerodimensional space without isolated points. Using Zorn's lemma, we can find a maximal family of clopen sets of the form $\{a_t\colon t\in T\}$ such that $T$ is a tree, $a_t \leq a_s$ when
	$s\leq t$ and $a_t\cap a_s = 0$ when $s$ and $t$ are incomparable. The algebra generated by this tree is countable, and is therefore an interval algebra. By Corollary \ref{interval}, it is enough to check that the elements $\{a_t\colon  t\in T\}$
	split all clopen subsets of $K$. So take $b$ a nonempty clopen set in $K$ that is not split by that family. Since we have no isolated points, we find two disjoint nonempty clopens $c,d\subseteq b$. Notice that $c,d\not\in\{a_t\colon  t\in T\}$ because
	they split $b$. The family $\{a_t\colon  t\in T\}\cup \{c,d\}$ would be a larger tree family, in contradiction with maximality.

($\impliedby$) Let $T$ be a Suslin tree, and let $[T]$ be the set of all (maximal) branches of $T$. For every $t$, let $a_t = \{x\in[T] : t\in x\}$, and let $\mathfrak{A}$ be the algebra of subsets of $[T]$ generated by the $a_t$. Since Suslin trees
are Aronszajn, we can use Theorem \ref{aronszajn}. It remains to show that $\mathfrak{A}$ is ccc. For this, notice that every nonempty element of $\mathfrak{A}$ contains nonempty element with atomic formula $a = \bigcap_{t\in R} a_t \setminus
\bigcup_{s\in S} a_s$. If we take a high enough node $r$ in a branch that belongs to $a$, then $a_r\subseteq a$. All this means that if there is an uncountable pairwise disjoint family in $\mathscr{A}$, then there is one made of elements of the form $a_r$, and that would give an uncountable family of pairwise incompatible elements of $T$, that contradicts that $T$ is Suslin.  
\end{proof}

\subsection{Ultraproducts of tunnels}
We assume some familiarity with the Banach space ultraproduct construction, as presented in \cite[Chapter~4]{aviles}, \cite{stern} or \cite{henson}.
Let $(K_i)_{i\in I}$ be a family of compact spaces indexed by $I$ and let $\mathscr U$ be a \emph{countably incomplete} ultrafilter on $I$. Then the Banach space ultraproduct $[C(K_i)]_\mathscr U$ is a Banach algebra under the product
$$
[(f_i)][(g_i)]=[(f_ig_i)].
$$
By general representation results, this algebra is isometrically isomorphic to one of the form $C(K)$ for some compact space $K$ which is called the topological ultracoproduct of the family $(K_i)$ following $\mathscr U$ and is denoted by $(K_i)^\mathscr U$. If all the $K_i$ coincide we speak of the ultracopower, instead.

\begin{prop}\label{prop:Ultra-have}
With the preceding notations, if each $K_i$ has a tunnel, then so $(K_i)^\mathscr U$ does.
\end{prop}

\begin{proof}
We need to translate our topological notions from $K$ to the algebra $C(K)$. The basic idea is that each open subset $U$ of $K$ gives rise to a closed ideal just taking
$$
J_U=C_0(U)=\{f\in C(K): f|_{K\backslash U}=0\},
$$
and, conversely, all closed ideals of $C(K)$ have this form. 
On the other hand, the condition $f|_{K\backslash U}=1$ is equivalent to the class of $f$ being the unit of the quotient algebra $C(K)/J_U$.

Thus, the fact that $U$ and $V$ are open subsets of $K$ with $\overline{U}\subseteq V$, which is obviously equivalent to the existence of $f\in C(K)$ such that $f|_U=0$ and $f|_{K\backslash V}=1$ can be stated as:
\begin{itemize}
\item There is $f\in C(K)$ such that $fg=0$ for all $g\in J_U$ and whose class is the unit of $C(K)/J_V$.
\end{itemize}

Now assume that each $K_i$ has a tunnel $\mathscr W_i$. We construct a family of open sets $\mathscr W$ of $(K_i)^\mathscr U$ as follows. For each $i\in I$ we pick $U_i\in \mathscr W_i$, then we consider the corresponding ideal $J_{U_i}\subseteq C(K_i)$ and form the ultraproduct $[J_{U_i}]_\mathscr U$. Quite clearly, $[J_{U_i}]_\mathscr U$ is a closed ideal in $[C(K_i)]_\mathscr U$ and, by the preceding remarks, this ideal determines a certain open set $W$ of $(K_i)^\mathscr U$. 
 Let us check that the family of open sets of $(K_i)^\mathscr U$ obtained in this way forms a tunnel.  

First, we prove that they form a chain. Take two families $(U_i), (V_i)$, with $U_i, V_i\in \mathscr W_i$ and let $W$ and $V$ be the corresponding subsets of $(K_i)^\mathscr U$. We partition $I$ into three subsets as follows:
\begin{itemize}
\item $I_{\natural}=\{i\in I: U_i=V_i\}$;
\item $I_{\flat}=\{i\in I: \overline{U_i}\subseteq V_i\}$;
\item $I_{\sharp}=\{i\in I: \overline{V_i}\subseteq U_i\}$.
\end{itemize}
Then exactly one of these sets belongs to $\mathscr U$.
If $I_{\natural}$ belongs to $\mathscr U$, then $W=V$. 
Now assume $I_{\flat}$ belongs to $\mathscr U$ and let us prove that $\overline{W}\subseteq V$. For each $i\in I_{\flat}$, take a continuous $f_i:K_i\to[0,1]$ such that 
$f_i|_{U_i}=0$ and $f_i|_{K_i\backslash V_i}=1$. If $i\notin I_{\flat}$, set $f_i=1$.
Let us take a look at $[(f_i)]$. It is clear that the class of $[(f_i)]$ in 
$$[C(K_i)]_\mathscr U\big{/}[J_{V_i}]_\mathscr U= [C(K_i){/}J_{V_i}]_\mathscr U $$
is the unit of the quotient algebra. On the other hand, if $(g_i)$ is a bounded family such that $g_i\in J_{U_i}$ for all $i\in I$, then $[(f_i)][(g_i)]= [(f_ig_i)]=0$ since $f_ig_i=0$ at least for $i\in I_{\flat}$. This shows that
$\overline{W}\subseteq V$. Finally, if $I_{\sharp}\in\mathscr U$, then $\overline{V}\subseteq W$.

To complete the proof we need to manage some points of $(K_i)^\mathscr U$, that is, some ``nice'' maximal ideals of $[C(K_i)]_\mathscr U$.

Let $(p_i)_i$ be a family such that $p_i\in K_i$ for each $i\in I$. Then we can define a unital homomorphism $[C(K_i)]_\mathscr U\to \mathbb R$ by the formula
\begin{equation}\label{eq:we agree}
[(f_i)]\longmapsto \lim_{\mathscr U(i)} f_i(p_i).
\end{equation}
The definition makes sense and, moreover, two families $(p_i)_i$ and $(q_i)_i$ induce the same homomorphism if and only if they represent the same element in the set-theoretic ultraproduct $\langle K_i\rangle_\mathscr U$, that is, when the set
$\{i\in I: p_i=q_i\}$ belongs to $\mathscr U$. If we agree to denote by $\langle (p_i)\rangle$ the  ``point'' of $(K_i)^\mathscr U$ associated to (\ref{eq:we agree}) (as in the Gelfands representation theorem, see e.g. \cite[Theorem
4.2.1]{Albiac-Kalton}), then we have an injective mapping $\langle K_i\rangle_\mathscr U\to (K_i)^\mathscr U$. This mapping is known to have dense range. We need a slightly stronger fact:
\medskip

\noindent{\bf Claim.} If, for each $i$, the set $D_i$ is dense in $K_i$, then every nonempty zero set of  $(K_i)^\mathscr U$ meets $\langle D_i\rangle_\mathscr U$. In particular, 
$\langle D_i\rangle_\mathscr U$ is dense in $(K_i)^\mathscr U$.
\medskip

\begin{proof}[Proof of the Claim]
The second assertion clearly follows from the first one since, by the very definition, $(K_i)^\mathscr U$ is a completely regular space.

So, let us check the first statement. The hypothesis  that $\mathscr U$ is countably incomplete is used as follows: there is a function $\delta\colon I\to(0,\infty)$ such that $\delta(i)\longrightarrow 0$ along $\mathscr U$. Now, take a non-negative, continuous
$f\colon (K_i)^\mathscr U\to \mathbb R$ with $f^{-1}(0)\neq \varnothing$. Write $f=[(f_i)]$, with $f_i\geq 0$ in $C(K_i)$ and put
$$
m(i)=\min_{x\in K_i} f_i(x) = \inf_{x \in D_i} f_i(x).
$$
Note that $m(i)\longrightarrow 0$ along $\mathscr U$ since otherwise $f$ would be invertible. Take $\delta$ as before and, for each $i\in I$, choose $x_i\in D_i$ so that $f_i(x_i)<m(i)+\delta(i)$. Then $f$ vanishes on the point $\langle(x_i)\rangle$ and the Claim is proved. 
\end{proof}

The proof will be complete if we show that if for each index $i$ the set $U_i$ is  open 
in $K_i$ and $p_i\in\partial U_i$, then $\langle (p_i)\rangle\in\partial W$, where $W$ is the open set of $(K_i)^\mathscr U$ attached to the family $(U_i)$ -- that is, to the ideal $[J_{U_i}]_\mathscr U$.

To do this we add two new entries to our basic dictionary: suppose $U$ is an open set in a compactum $K$ and that $p\in K$. Then:
\begin{itemize}
\item $p\notin U$ is equivalent to the statement ``for every $f\in J_U$ one has $f(p)=0$''.
\item $p\in \overline{U}$ is equivalent to the statement ``for $g\in C(K)$ such that $g(p)=1$ there is $q\in U$ such that $g(q)\geq {1\over 2}$''.
\end{itemize}

Now, if $p_i, U_i, W$ are as before, then $p_i\in\overline{U_i}\setminus U_i$.
Since every $f\in J_W$ can be written as $[(f_i)]$, with $f_i\in J_{U_i}$, we have
$$
f(\langle (p_i)\rangle)= \lim_{\mathscr U(i)}f_i(p_i)=0,
$$
so, certainly, $\langle (p_i)\rangle\notin W$.

Finally, we check that  $\langle (p_i)\rangle\in \overline{W}$. We first remark that $W$ contains every point of the form $\langle (q_i)\rangle$ with $q_i\in U_i$ for all $i\in I$ (think of functions $f_i$ such that $f_i(q_i)=1$ and $f_i\in J_{U_i}$). Now, 
if $f$ is a continuous function on $(K_i)^\mathscr U$ such that $f(\langle (p_i)\rangle)= 1$, then writing $f=[(f_i)]$ and recalling that $p_i\in \overline{U_i}$ we can pick $q_i$ such that $f_i(q_i)\geq {1\over 2} f_i(p_i)$, so $f(\langle (q_i)\rangle)\geq  {1\over 2}$, so $\langle (p_i)\rangle\in \overline{W}$.
\end{proof}

We now give the application that motivated our interest in ultracoproducts. We need some basic facts from model theory in the context of Banach spaces.
The reader can take a look to \cite{stern} or \cite{henson} for the general background and to \cite{bankston} for a more topological approach.

 Following the uses in model theory,
let us say that two compact spaces $K$ and $L$ are co-elementarily equivalent if there are ultrafilters $\mathscr U$ and $\mathscr V$ such that $K^\mathscr U$ and $L^\mathscr V$ are homeomorphic, equivalently, the Banach algebras $(C(K))_\mathscr U$ and 
$(C(L))_\mathscr V$ are isomorphic. This happens if and only if  the underlying Banach spaces $(C(K))_\mathscr U$ and 
$(C(L))_\mathscr V$ are (linearly) isometric. This roughly means that the base Banach spaces $C(K)$ and $C(L)$ ``approximately'' satisfy the same positive bounded sentences (in a suitable signature); see \cite[Chapter 5]{henson}.

Thus, the following result explains in part why it is so difficult to find compacta (without isolated points) having no tunnel.

\begin{prop}
Let $K$ be a compact space. The following conditions are equivalent:
\begin{itemize}
\item[(a)] $K$ has no isolated points.
\item[(b)] There is an ultrafilter $\mathscr U$ (on some index set) for which the ultracopower $K^\mathscr U$ has a tunnel.
\item[(c)] $K$ is co-elementarily equivalent to a compactum having a tunnel.
\end{itemize}
\end{prop}

\begin{proof}
We first remark that the property of (not) having isolated point is preserved under co-elementary equivalence. This is implied by the following two facts:
\begin{itemize}
\item A compact space $K$ has an isolated point if and only if the algebra $C(K)$ has a ``minimal idempotent'': a non-zero $f\in C(K)$ such that $f^2=f$ with the property that if $0\leq g\leq f$, then $g=cf$ for some $c\in\mathbb R$. 
\item Every idempotent in $[C(K_i)]_\mathscr U$ can be written as $[(f_i)]$, where $f_i$ is an idempotent of $C(K_i)$.
\end{itemize}

We thus have (c)$\implies$(a). Next, we prove (a)$\implies$(b). The key fact is that every compactum is co-elementarily equivalent to some metrizable compact space. This follows from the Banach space version of the  (``downward'') L\"owenheim--Skolem theorem; see \cite[Theorem 2.2]{stern} or \cite[9.13 Proposition]{henson}. Suppose $K$ has no isolated points and let $M$ be a metrizable compactum such that $K^\mathscr U$ and $M^\mathscr V$ are homeomorphic, where $\mathscr U$ and $\mathscr V$ are ultrafilters on suitably chosen sets of indices. By the preceding remark, neither $K^\mathscr U$ nor $M$ have isolated points. By Corollary~\ref{cor:metriz-have}, $M$ has a tunnel and so $M^\mathscr V$ and $K^\mathscr U$ have, by Proposition~\ref{prop:Ultra-have}.

(b)$\implies$(c) Each compact space is co-elementarily equivalent to its ultracopowers, by the Banach space version of the Keisler--Shelah (``ultrapower'') theorem; see \cite[Theorem 2.1]{stern} or \cite[10.7 Theorem]{henson}.
\end{proof}

\section{Twisted sums} \label{derivation2}
Recall that a short exact sequence is a diagram of Banach spaces and (linear, bounded) operators
\begin{equation}\label{diag:yzx}
\begin{CD}
0@>>> Y @>\imath>> Z @>\pi>> X@>>> 0
\end{CD}
\end{equation}
in which the kernel of each arrow agrees with the range of the preceding one.
We say that a short exact sequence is trivial
if there is an operator $\varpi\colon Z\to Y$ such that $\varpi\circ\imath={\bf I}_Y$ (or, equivalently, there is an operator $\jmath\colon X\to Z$ such that $\pi\circ\jmath={\bf I}_Z$). Note that (\ref{diag:yzx}) is trivial if and only if 
$\imath[Y]$ is a complemented subspace of $Z$. In this case the space $Z$ is linearly homeomorphic to the direct sum $Y\oplus X$. Simple examples show that the converse is not true.

The space $C(K)$ can be viewed as a subspace (or a subalgebra) of $\ell_\infty(K)$,
the Banach algebra of all bounded functions $f\colon K\to \mathbb{R}$, again with the  supremum norm. 

Given a family $\mathscr A$ of subsets of $K$, we define an intermediate space $C(K) \subseteq  X(\mathscr A) \subseteq \ell_\infty(K)$ as the Banach space generated by $C(K)$ and by the characteristic functions of the sets in $A$. This produces a short exact sequence
\begin{equation}\label{CKXA}
\begin{CD}
0@>>> C(K) @>\imath >> X(\mathscr A) @>\pi >> X(\mathscr A)/C(K) @>>> 0,
\end{CD}
\end{equation}
in which $\imath$ is the inclusion map and $\pi$ is the natural quotient map.
This was the approach followed by Amir \cite{amir}. For this sequence to provide a relevant example, we must  ensure that it is not trivial, and we must identify what the quotient $X(\mathscr A)/C(K)$ is. Lemmas \ref{quotientc0} and \ref{derivationenough} will deal with these issues.

Recall that the oscillation of a  function $f\colon K \longrightarrow \mathbb{R}$ at a point $x\in K$ is defined by
	$$\osc f(x) = \inf_V\sup_{y,z\in V} \big{(}f(y)-f(z) \big{)},$$
	where $V$ runs over the neighborhoods of $x$ in $K$. The oscillation of an arbitrary function $f$ on $K$ is the number  $\mathrm{osc} f = \sup_{x\in K} \mathrm{osc} f(x)$.

\begin{lem}\label{quotientc0} Let $\mathscr A$ be a family of subsets of $K$.
	If  the boundaries of sets in $\mathscr A$ are all nonempty and pairwise disjoint, then the quotient space $X(\mathscr A)/C(K)$ is isometric to $c_0(\mathscr A)$.
\end{lem}
\begin{proof}
	
	The equivalence classes of the characteristic functions $1_A$ for $A\in\mathscr{A}$ generate the quotient space $X(\mathscr A)/C(K)$. We check that these vectors are isometric to the basis of $c_0$ multiplied by $1/2$. That is, we want to show that
	\begin{equation}
	\label{norm}
	\left\| \sum_{i=1}^n\lambda_i 1_{A_i} + C(K) \right\|_{X(\mathscr A)/C(K)} = \frac{1}{2}\,\max_{1\leq i\leq n}|\lambda_i|
	\end{equation}
	whenever $A_i\in \mathscr{A}$ and $\lambda_i\in\mathbb{R}$.
	 The norm of (the class of) a function $f$ in the quotient space by $C(K)$ is the distance of $f$ to $C(K)$, which by a classical result in topology (see e.g. \cite[Proposition 1.18]{Benyamini}) equals
	half of the oscillation of $f$, so
	$$\big{\|}f\big{\|}_{X(\mathscr A)/C(K)} = \dist(f,C(K)) = \frac{\osc f}{2}.$$
	A characteristic function $1_{A}$ has oscillation 1 at every point of $\partial A$ while it is continuous (oscillation 0) out of $\partial A$. Since the sets of $\mathscr A$ have disjoint nonempty boundaries, a linear combination $f = \sum_i \lambda_i 1_{A_i}$ has oscillation $|\lambda_i|$ on $\partial A_i$ and oscillation 0 out of these boundaries. From this, equation (\ref{norm}) follows.
\end{proof}

We now describe a derivation procedure induced by $\mathscr A$ that will help in proving that the short exact sequence (\ref{CKXA}) is not trivial. This is based on an idea of Ditor \cite{Ditor}. Suppose again that the subsets of $\mathscr A$ have disjoint boundaries. 

\begin{defi}\label{derivation}
	Given $D\subseteq \bigcup_{A\in\mathscr A}\partial A$, we define $D_{(1)}$ as the set of those $a\in D$ for which the following is true: If $a\in\partial A$ and $V$ is a neighborhood of $a$, then there are $B, C\in\mathscr A$, both different from $A$,  such that
	$\partial B\cap V\cap A\cap D\neq \varnothing$ and 
	$\partial C\cap V\cap A^c \cap D \neq \varnothing$.
	
\end{defi}

For every $n\in\mathbb{N}$ we recursively define $D_{(n)}=(D_{(n-1)})_{(1)}$, starting from $D_{(0)} = D$. Recall that if $Y$ is a subspace of a Banach space $Z$, then by a \emph{linear lifting} of  the quotient map $\pi\colon Z \longrightarrow Z/Y$ we understand a
(not necessarily bounded) linear right-inverse of $\pi$. The quotient map $\pi\colon Z \longrightarrow Z/Y$ admits a bounded linear lifting $\jmath$ if and only if $Y$ is complemented in $Z$, the map $\varpi(z)=z-\jmath(\pi(z))$ being a bounded projection of $Z$ onto $Y$.


\begin{lem}\label{derivationenough}
	Let $\mathscr A$ be a family of subsets of $K$ with nonempty pairwise disjoint boundaries. Let $D = \bigcup\{\partial A\colon  A\in\mathscr{A}\}$.  If $D_{(n)}\neq\varnothing$, then the norm of any linear lifting for the quotient map $\pi\colon X(\mathscr
	A)\longrightarrow X(\mathscr A)/C(K)$ is at least $n$. Therefore, if $D_{(n)}\neq\varnothing$ for every $n\in \mathbb N$, then $C(K)$ is uncomplemented in $X(\mathscr A)$.
\end{lem}

\begin{proof}
	We will first prove the following claim.
	\medskip

\noindent\textbf{Claim.} Let $(f_A)$ be any family in $X(\A)$ such that $g_A := f_A - 1_A \in C(K)$ for every $A\in\A$. Suppose $D_{(n)}\neq\emptyset$ and fix any $\varepsilon>0$.
	Then there are different sets $A_1,\dots, A_n\in\A$ and signs $u_i=\pm 1$ such that
	\begin{equation}\label{halfn}
	\left\|\sum_{i=1}^n u_if_{A_i}\right\|_\infty\geq \frac{n}{2}-n\varepsilon.
	\end{equation}
To obtain this we will prove by induction on $1\leq k\leq n$ that there are different sets $A_1,\dots, A_k$, signs $u_i=\pm1$, and a point
	\begin{equation}\label{ak}
	a_k\in D_{(n-k)}\big{\backslash }\left( \bigcup_{i=1}^k\partial A_i  \right)
	\end{equation}
	such that
	\begin{equation}\label{value}
	(u_1f_{A_1}+\cdots+u_1f_{A_k})(a_k)> \frac{k}{2}-k\varepsilon.
	\end{equation}
	
	First we deal with the case $k=1$. Pick any
	$a\in D_{(n)}$ and let $A_1$ be the set from $\A$ whose boundary contains $a$. Consider the neighborhood of $a$,
	$$
V=\{z\in K\colon |g_{A_1}(z)-g_{A_1}(a)|<\varepsilon\}.
	$$
	Using Definition~\ref{derivation}, we can find $x\in V\cap A_1\cap D_{(n-1)}
	$ and $y\in V\cap A_1^c\cap D_{(n-1)}$ that do not belong to $\partial A_1$. We have two cases:
	\begin{itemize}
		\item If $g_{A_1}(a)\geq -1/2$, then $f_{A_1}(x)>{1\over 2}-\varepsilon$ and we take $a_1=x$ and $u_1=1$.
		\item If $g_{A_1}(a)\leq -1/2$, then $f_{A_1}(y)<\varepsilon-{1\over 2}$ and we take $a_1=y$ and $u_1=-1$.
	\end{itemize}
Let us check the induction step. Suppose one has found $A_i, u_i$ for $1\leq i\leq k$  and $a_k$ satisfying (\ref{ak}) and (\ref{value}). The point $a_k$ does not belong to any border $\partial A_i$ by (\ref{ak}), each function $1_{A_i}$ is continuous at $a_k$, and hence also each function $f_{A_i} = g_{A_i} + 1_{A_i}$ is continuous at $a_k$. Since $u_1f_{A_1}+\cdots+u_1f_{A_k}$ is continuous at $a_k$ we can find a neighborhood $V$ of $a_k$ disjoint from $\bigcup_{1\leq i\leq k}\partial A_i$ and  such that the value of $u_1f_{A_1}+\cdots+u_1f_{A_k}$ at any point of $V$ differs from 
	$(u_1f_{A_1}+\cdots+u_1f_{A_k})(a_k)$ at most by $\varepsilon/2$. Let $A$ be the set of $\A$ whose boundary contains $a_k$.
	Shrinking $V$ if necessary, we may also assume that
	$$
	|g_A(a_k)-g_A(x)|<\varepsilon/2\ \ \text{ for all } x\in V.
	$$
According to Definition~\ref{derivation} there are $B,C\in\A$ different from $A$ such that 
	\begin{equation}\label{B}
	\partial B\cap V\cap A\cap D_{(n-k-1)}\neq \varnothing,
	\end{equation}
	\begin{equation}\label{C}
	\partial C \cap V\cap A^c\cap D_{(n-k-1)}\neq \varnothing.
	\end{equation}
If $g_A(a_k)\geq -{1\over 2}$, then we take $A_{k+1}=A, u_{k+1}=1$ and $a_{k+1}$ any element in the set (\ref{B}), and we get
\begin{align*}
\big{(}u_1f_{A_1}&+\cdots+u_k f_{A_k}+u_{k+1}f_{A_{k+1}}\big{)}(a_{k+1})\\
&= (u_1f_{A_1}+\cdots+u_1f_{A_k})(a_{k+1})+f_{A}(a_{k+1}) \\
		&>  (u_1f_{A_1}+\cdots+u_1f_{A_k})(a_{k})-\frac{\varepsilon}{2}+g_{A}(a_{k+1})+1\\
		&> \frac{k}{2}-k\varepsilon- \frac{\varepsilon}{2} + g_A(a_k) -\frac{\varepsilon}{2} + 1\\
		&\geq \frac{k}{2}-k\varepsilon- \frac{\varepsilon}{2}-\frac{1}{2}  -\frac{\varepsilon}{2} + 1\\ &= \frac{k+1}{2} - (k+1)\varepsilon.
\end{align*}
The other case is that $g_A(a_k)< -{1\over 2}$. Then we take $A_{k+1}=A, u_{k+1}=-1$ and $a_{k+1}$ in the set (\ref{C}) and we get
\begin{align*}
\big{(}u_1f_{A_1}&+\cdots+u_1f_{A_k}+u_{k+1}f_{A_{k+1}})(a_{k+1}\big{)}\\
&= (u_1f_{A_1}+\cdots+u_1f_{A_k})(a_{k+1})-f_{A}(a_{k+1}) \\
		&>  (u_1f_{A_1}+\cdots+u_1f_{A_k})(a_{k})-\frac{\varepsilon}{2}-g_{A}(a_{k+1})\\
		&> \frac{k}{2}-k\varepsilon- \frac{\varepsilon}{2} - g_A(a_k) -\frac{\varepsilon}{2}\\
		&> \frac{k}{2}-k\varepsilon- \frac{\varepsilon}{2}+\frac{1}{2}  -\frac{\varepsilon}{2}\\ &=\frac{k+1}{2} - (k+1)\varepsilon.
		\end{align*}
	This finishes the proof of the claim.  
	
	\medskip
	
	If  $L\colon X(\A)/C(K)\longrightarrow X(\A)$ is a linear lifting for the quotient map $\pi\colon  X(\A)\longrightarrow X(\A)/C(K)$, then the functions $f_A=L(1_A+C(K)))$ satisfy that $f_A-1_A\in C(K)$. Therefore (\ref{halfn}) holds for some $u_i$ and $A_i$, and hence 
	\begin{eqnarray*}
		\frac{n}{2}-n\varepsilon &\leq& \left\|\sum_{i=1}^n u_if_{A_i}\right\|_\infty
		=\left\|L\left(\sum_{i=1}^n u_i(1_{A_i}+C(K))\right)\right\|_\infty\\
		&\leq& \|L\|\cdot \left\|\sum_{i=1}^n u_i(1_{A_i}+C(K))\right\|_{X(\A)/C(K)}\\
		&=& \|L\|\cdot\frac{1}{2}\max_{1\leq i\leq n}|u_i| = \frac{\|L\|}{2}.
	\end{eqnarray*}
	Since this holds for arbitrary $\varepsilon$ we conclude that $\|L\|\geq n$.
\end{proof}



\begin{cor}\label{derivation1enough}
	Let $\mathscr A$ be a family of subsets of $K$ with nonempty pairwise disjoint boundaries such that 
	$$\partial A \subseteq \overline{\{\partial B\cap A\colon B\in\mathscr{A}\setminus\{A\}\}} \cap  \overline{\{\partial C\setminus A\colon C\in\mathscr{A}\setminus\{A\}\}}$$
for all $A\in\mathscr{A}$. Then the exact sequence 
	\begin{equation*}
	\begin{CD}
	0@>>> C(K) @>>> X(\mathscr A) @>>> c_0(\mathscr A)@>>> 0
	\end{CD}
	\end{equation*}
	is not trivial.
\end{cor}

Recall that an open set $V$ is said to be \emph{regular} if it is the interior of its closure, or equivalently if $\partial V = \partial\overline{V}$.

\begin{defi}
	A regular tunnel is one made of regular open sets.
\end{defi}

\begin{thm}\label{twistedtunnel}
	If $K$ has a regular tunnel of cardinality $\kappa$, then there exists nontrivial exact sequence
	\begin{equation}\label{c0CK}\begin{CD}
	0@>>> C(K) @>>> Z @>>> c_0(\kappa)@>>> 0.
	\end{CD} 
	\end{equation}
	
\end{thm} 

\begin{proof}
	We can suppose that the tunnel is nontrivial, in the sense that all borders are nonempty. We have to prove that such a tunnel $\mathscr A$ satisfies the hypotheses of Corollary~\ref{derivation1enough}. Fix $A\in\A$, and $x\in \partial A$, and $U$ a neighborhood of $x$. Since $x\in\partial A$, we have that $U\cap A$ is a nonempty open set, so by the definition of tunnel there exists $B\in\A$ such that $\partial B \cap U\cap A\neq\varnothing$. On the other hand, since $\A$ is a regular tunnel, $x\in\partial A = \partial \overline{A} = \partial(K\setminus \overline{A})$, therefore $U\setminus \overline{A}\neq\varnothing$, and again there exists $C\in\A$ such that $\partial C \cap U\setminus \overline{A}\neq\varnothing$. 
\end{proof}

Now we are ready to prove the theorem announced in the Introduction.

\begin{thm}\label{the-main}
If there is a splitting chain of clopen sets in $\omega^\ast$, then there is a nontrivial exact sequence
\begin{equation*}
\begin{CD}
	0@>>> C(\omega^\ast) @>>> Z @>>> c_0(\mathfrak{c})@>>> 0.
	\end{CD}
	\end{equation*}
\end{thm}

\begin{proof}
	The \v{C}ech-Stone remainder $\omega^\ast$ has the property that every open $F_\sigma$-set is regular (this is a consequence of the fact that every nonempty $G_\delta$ closed set has nonempty interior). Thus, if $\mathscr{U}$ is a splitting chain
	of clopen subsets of $\omega^\ast$, the countable increasing unions form a tunnel (by Proposition~\ref{forantonio}) which is moreover regular. Its cardinality is $|\mathscr{U}|^\omega = \mathfrak{c}$.
\end{proof}	

As we will see in the next section, splitting chains of clopen sets do exist in $\omega^\ast$ under some assumptions, e.g., under $\mathsf{CH}$ or in the classical Cohen model, but one cannot prove their existence in $\mathsf{ZFC}$. We do not know if
twisted sums like in Theorem \ref{the-main} exist in $\mathsf{ZFC}$.  

In \cite{corr}, a twisted sum as above was constructed under $\mathsf{CH}$. It was used to produce a further nontrivial twisted sum
\begin{equation}\label{co}
\begin{CD}
0@>>>C(\omega^\ast) @>>> Z @>>> C(\omega^\ast) @>>> 0.
\end{CD}
\end{equation}
We do not know if such a sequence exists in any axiomatic setting other than $\mathsf{CH}$. The note \cite{corr} pointed out that a statement made in \cite{avilesAiM} that all exact sequences like (\ref{co}) are trivial was incorrect. It is
also unknown if there are nontrivial sequences
$$
\begin{CD}
0@>>>C(\omega^\ast) @>>> Z @>>> X @>>> 0,
\end{CD}
$$
with $X$ of density less than $\mathfrak{c}$. It cannot be taken separable because $C(\omega^\ast)$ is ``separably injective''; see \cite[Section 2.5]{aviles} for this issue. We remark that if  (\ref{sex}) is nontrivial, then so is the ``expanded'' sequence
\begin{equation*}
\begin{CD}
0@>>> Y @>(\imath,0)>> Z\oplus S @>\pi\times{\bf I}_S>> X\oplus S@>>> 0, 
\end{CD}
\end{equation*}
where, as one can guess, $(\imath,0)(y)=(\imath(y),0)$ and $(\pi\times{\bf I}_S)(y,s)=(\pi(y),s)$, whichever is the space $S$.
 In particular, any reduction of the size of $c_0(\kappa)$ on the right side of our exact sequences would be an improvement of our statements.

We finish this section with some results relating tunnel-like conditions with the existence of twisted sums in settings different than that of $\omega^\ast$. 

\begin{prop}\label{ordertwisted}
Suppose that there exists a continuous surjection $f\colon K\longrightarrow L$, where $L$ is a linearly ordered space, and a set $S\subseteq L$ such that
$$f^{-1}(s) \subseteq \overline{\{x\colon  s<f(x)\in S \}} \cap \overline{\{x\colon  s>f(x)\in S\}}  $$
for every $s\in S$. Then there is a nontrivial sequence
\begin{equation*}\begin{CD}
0@>>> C(K) @>>> Z @>>> c_0(S) @>>> 0.
\end{CD}
\end{equation*}
\end{prop}

\begin{proof} For $s\in S$, put $A_s=f^{-1}(-\infty,s)$ and then 
	consider the family $\mathscr{A} = \{ A_s\colon s\in S\}$. Notice that $f^{-1}(s) \subseteq \partial A_s$ by the assumption and $\partial A_s \subseteq f^{-1}(s)$ by continuity
	of $f$. Thus, $\partial A_s = f^{-1}(s)$ and we can apply Corollary~\ref{derivation1enough}.
\end{proof}

Proposition~\ref{ordertwisted} unifies a number of earlier constructions of twisted sums with $C(K)$-spaces. Indeed, if we take $K=L=[0,1]$ with the usual order, $f$ the identity, and $S$ is the set of dyadic rationals, one gets the Foia\c s--Singer sequence in \cite[Theorems 3 and 4]{f-s}, in which the space $\Gamma[0,1]$ corresponds to ``our'' $X(\mathscr A)$. Analogously, taking $K=L=\{0,1\}^\mathbb N$ as the Cantor set with the ``lexicographical'' order, $f$ the identity, and $S$ is the subset of sequences with finitely many ones, one obtains the exact sequence used in \cite{corr}.

Finally, a nontrivial sequence of the form
$$
\begin{CD}
0@>>>C(\omega^\omega) @>>> Z @>>> c_0 @>>> 0
\end{CD}
$$
(see \cite[Section~4]{Felix03}) can be obtained from Lemma~\ref{derivationenough} as follows. We need the following representation of $\omega^\omega$.
We consider a \emph{reversed, signed version of Schreier family:}
$$L = \left\{ \left(\frac{1}{n_1},\ldots,\frac{1}{n_k},0,0,0,,\ldots\right) : n_i\in \mathbb{Z}, k\leq |n_1|<|n_2|<\cdots<|n_k| \right\}$$ 
We put on $L$ the lexicographical order, declaring $r<s$ if $r_i<s_i$, where $i$ is the first index such that $r_i\neq s_i$. The line  $L$ is compact in the order topology. This is because the order is complete: every subset has an infimum and a supremum. Moreover, $L$ is countable, hence scattered. The derivatives can be checked to be the sets
$$L^{(d)} = \left\{ \left(\frac{1}{n_1},\ldots,\frac{1}{n_k},0,0,0,,\ldots\right) \in L :  k\leq |n_1|-d   \right\}.$$
Thus, $L$ has height $\omega$ and $L^{(\omega)}$ is a singleton, so $L$ is homeomorphic to the ordinal interval $[0,\omega^\omega]$.

Now, for $s\in L^{(1)}$, put $A_s=(-\infty,s)$ and define $\mathscr A=\{A_s:s\in L^{(1)}\}$. Then if $D=\bigcup_s \partial A_s$ we have $D=L^{(1)}$ and by the peculiarities of the ordering we have $D_{(n)}=L^{(n+1)}\neq \varnothing$ for all $n$ and Lemma~\ref{derivationenough} shows that the sequence
$$
\begin{CD}
0@>>>C(\omega^\omega) @>>> X(\mathscr A) @>>> c_0 @>>> 0
\end{CD}
$$
is not trivial.

\section{Splitting chains in $\mathscr{P}(\omega)/\fin$}\label{forcing}

A set $A\subseteq \omega$ \emph{splits} $B\subseteq \omega$ if $|A\cap B| = |B\setminus A| = \omega$. We say that a family $\mathscr{A}\subseteq [\omega]^\omega$ is \emph{splitting} if for each $B\in [\omega]^\omega$ there is $A\in \mathscr{A}$
splitting $B$. Clearly, $\omega^*$ has a splitting chain of clopens if and only if there is a family $\mathscr{C}\subseteq [\omega]^\omega$ which  is splitting and which forms a $\subseteq^*$-chain. We will call such $\mathscr{C}$ a \emph{splitting
chain}. Consequently, if there is a splitting chain in $[\omega]^\omega$, then $\omega^*$ has a tunnel. It is unclear if this implication can be reversed (see Remark \ref{clopen}). 

In this section we will consider the question of the existence of splitting chains. 

We will begin with the easy observation that the existence of a splitting chain is consistent with $\mathsf{ZFC}$. 
Recall that $(\mathscr{L},\mathscr{R})$ is \emph{a pre-gap} if both $\mathscr{L}$ and $\mathscr{R}$ are lienarly ordered by $\subseteq^*$, and $L\subsetneq^* R$ for each $L\in \mathscr{L}$ and $R\in \mathscr{R}$. We will assume that both
$\mathscr{L}$ and $\mathscr{R}$ are nonempty. If $C\subseteq \omega$ is
such that $L\subseteq^* C$ and $C\subseteq^* R$ for each $L\in \mathscr{L}$ and $R\in \mathscr{R}$, then we say that $C$ \emph{interpolates} $(\mathscr{L}, \mathscr{R})$. If there is no $C$ interpolating $(\mathscr{L},\mathscr{R})$, then we say that
$(\mathscr{L},\mathscr{R})$ is \emph{a gap}. 
\begin{prop}\label{CH}(See \cite[Theorem 2.4]{Nyikos}.)
	$\mathsf{(CH)}$ implies that there is a splitting chain.  
\end{prop}
\begin{proof}
	Enumerate $[\omega]^\omega = \{A_\alpha\colon \alpha<\omega_1\}$. We will construct the desired chain inductively. Let $\mathscr{C}_0 = \varnothing$ and suppose that we have constructed $\mathscr{C}_\alpha$ for each
	$\alpha\leq \gamma$ in such a way that $\mathscr{C}_\alpha \subseteq \mathscr{C}_\beta$ for each $\alpha<\beta\leq \gamma$ and $A_\alpha$ is split by an element of $\mathscr{C}_{\alpha+1}$ for $\alpha<\gamma$. 
	If $A_\gamma$ is split by $\mathscr{C}_\gamma$, then put $\mathscr{C}_{\gamma+1} = \mathscr{C}_\gamma$. Otherwise, let $\mathscr{C}^\star = \{C\in \mathscr{C}_\gamma \colon A_\gamma \subseteq^* C\}$ and $\mathscr{C}_\star = \{C\in
		\mathscr{C}_\gamma \colon A_\gamma \cap C =^* \varnothing \}$ and notice that $\mathscr{C} = \mathscr{C}^\star \cup \mathscr{C}_\star$. Since there are no $(\omega,\omega)$-gaps, there is an infinite $N$ such that $C_0 \subseteq^* N \subseteq^* C_1$ for each 
	$C_0\in \mathscr{C}_\star$ and $C_1\in \mathscr{C}^\star$. Fix any $H \subseteq A_\gamma$ splitting $A_\gamma$. If $N$ splits $A_\gamma$ then let $\mathscr{C}_{\gamma+1} = \mathscr{C}_\gamma \cup \{N\}$. If $N \cap
	A_\gamma =^* \varnothing$, then let $\mathscr{C}_{\gamma+1} = \mathscr{C}_\gamma \cup \{N\cup H\}$. Finally, if $A_\gamma \subseteq^* N$, then let $\mathscr{C}_{\gamma+1} = \mathscr{C}_\gamma \cup \{N\setminus H\}$.

	If $\gamma$ is a limit ordinal, then let $\mathscr{C}_\gamma = \bigcup_{\alpha<\gamma} \mathscr{C}_\alpha$.
\end{proof}

We say that $(\mathscr{L},\mathscr{R})$ is \emph{a cut in} a $\subseteq^*$-chain $\mathscr{C}$ if it is a pre-gap, $\mathscr{L}\cup \mathscr{R}\subseteq \mathscr{C}$ and there is no element of $\mathscr{C}$
interpolating it. In other words, if $C\in \mathscr{C}$,
then either $R\subseteq^* C$ for some $R\in \mathscr{R}$ or $C\subseteq^* L$ for some $L\in \mathscr{L}$.

We call a pre-gap $(\mathscr{L}, \mathscr{R})$ \emph{tight} if for each infinite $A\subseteq \omega$ such that $A\subseteq^* R$ for each $R\in \mathscr{R}$, there is $L\in \mathscr{L}$
such that $L\cap A\ne^* \varnothing$. We say that $A\subseteq\omega$ \emph{spreads a pre-gap} $(\mathscr{L},\mathscr{R})$ if $A\cap L=^*
\varnothing$ and $A\subseteq^* R$ for each $L\in \mathscr{L}$, $R\in \mathscr{R}$. In other words, $A$ spreads a pre-gap $\mathscr{G}$ if and only if it witnesses
that $\mathscr{G}$ is not tight. We say that a cut in a chain is tight if it is tight as a pre-gap.

\begin{prop}\label{splitting-tight} A $\subseteq^*$-chain $\mathscr{C}$ is splitting if and only if every cut in $\mathscr{C}$ is tight. 
\end{prop}
\begin{proof} Assume that $(\mathscr{L},\mathscr{R})$ is a cut in $\mathscr{C}$ which is not tight. It means that there is an infinite $A\subseteq \omega$ such that $A\subseteq^* R$ and $A\cap L=^* \varnothing$ for each $L\in \mathscr{L}$, $R\in
	\mathscr{R}$. If $C\in \mathscr{C}$, then either $R\subseteq^* C$ for some $R\in \mathscr{R}$ or $C\subseteq^* L$ for some $L\in \mathscr{L}$. In both cases $C$ does not split $A$.

	If an infinite $A$ is not split by $\mathscr{C}$, then let $\mathscr{L} = \{C\in \mathscr{C}\colon C\cap A=^* \varnothing \}$ and
		$\mathscr{R} = \{C\in \mathscr{C}\colon A\subseteq^*C\}$. It is easy to verify that $(\mathscr{L}, \mathscr{R})$ is a cut
		in $\mathscr{C}$ and $A$ witnesses that it is not tight.
\end{proof}

We say that a pre-gap $(\mathscr{L}, \mathscr{R})$ is of type $(\kappa,\lambda)$ if $\mathscr{L}$ is of cofinality $\lambda$ and $\mathscr{R}$ is of coinitiality $\kappa$ (with respect to the $\subseteq^*$ order).

\begin{thm}[\cite{Nyikos83}]\label{nyikos}
There is a $(\omega_1, \omega_1)$-tight pre-gap if and only if $\mathfrak{p}=\omega_1$.
\end{thm}

The above result indicates that it is quite difficult to construct a splitting chain in general. It is not completely trivial to obtain a single tight pre-gap, and a splitting chain has to look like a tight pre-gap \emph{everywhere}. Also, Theorem \ref{nyikos} 
suggests a strategy to prove that consistently there are no splitting chains. We have to find a model in which $\mathfrak{p}=\omega_2$ but
for some reasons every chain has to have a cut being an $(\omega_1, \omega_1)$-gap. As we will see such a
reason can be provided by Proper Forcing Axiom. The following theorem was proved by Nyikos in \cite{Nyikos}.

\begin{thm}\label{PFA} Assume $\mathsf{MA}(\omega_1)$ holds, $\mathfrak{c}=\omega_2$ and there are no $(\omega_1,\omega_2)$-gaps. Then there is no splitting chain. 
\end{thm}

\begin{proof}
	Assume $\mathscr{C}$ is a $\subseteq^*$-chain of infinite subsets of $\omega$ and assume that $\omega \in \mathscr{C}$. If $\mathscr{C}$ is splitting, then we can find an increasing sequence $(L_\alpha)_{\alpha<\omega_1}$ of elements of $\mathscr{C}$. $\mathsf{MA}(\omega_1)$ implies
	that if $L_\alpha\subseteq^* B$ for some $B$, then there is $B'\subsetneq^* B$ such that $L_\alpha \subseteq^* B'$ for each $\alpha<\omega_1$. Using this remark one can construct inductively a $\subseteq^*$-decreasing sequence
	$(R_\alpha)_{\alpha<\kappa}$ of elements of $\mathscr{C}$ such that $(L_\alpha,R_\beta)_{\alpha<\omega_1, \beta<\kappa}$ forms a gap. Indeed, assume that $R_0 =\omega$ and suppose that $(L_\alpha, R_\beta)_{\alpha<\omega_1, \beta<\gamma}$ is not a gap. Then there is $B\subseteq \omega$
	interpolating it.
	Using, the above remark, we can find  $B'\subsetneq^* B$ such that $L_\alpha\subseteq^* B'$ for each $\alpha<\omega_1$. There is $C\in \mathscr{C}$ splitting $B\setminus B'$. Clearly, $C$ has to interpolate $(L_\alpha, R_\beta)_{\alpha<\omega_1, \beta<\gamma}$ and thus,
	we can proceed with the construction. 

	Now, by our assumption $\gamma=\omega_1$. By Theorem \ref{nyikos} $(L_\alpha,R_\alpha)_{\alpha<\omega_1}$ cannot be tight. Thus, there is $A\subseteq \omega$ 
	such that $L_\alpha \cap A =^* \varnothing$ and $A\subseteq^* R_\alpha$ for every $\alpha<\omega_1$. So, $A$ is not split by $\mathscr{C}$. 
\end{proof}

\begin{cor}
$\mathsf{PFA}$ implies that there is no splitting chain. \end{cor}

\begin{proof} $\mathsf{PFA}$ implies that the assumptions of Theorem \ref{PFA} are satisfied, see e.g. \cite[Theorem 8.6]{Todorcevic89}.
\end{proof}

\begin{remark} Of course Theorem~\ref{PFA} means that there is no splitting chain of clopens in $\omega^*$. In \cite{Nyikos} Nyikos proved a stronger theorem: under the assumptions of Theorem~\ref{PFA} the space $\omega^*$ does not have a tunnel (cf.\ Remark~\ref{clopen}).
\end{remark}

In the light of the above results it is natural to ask if the existence of a splitting chain implies $\mathsf{CH}$. First, notice that
the proof of Proposition \ref{CH} uses $\mathsf{CH}$ in an essential way. The reason is that, by the classical result of Hausdorff, there are $(\omega_1, \omega_1)$-gaps in $\mathsf{ZFC}$. So, to construct a splitting chain by a transfinite
induction longer than $\omega_1$ we would have to keep control on the cuts appearing in the construction at steps of cofinality $\omega_1$ to avoid a situation in which we have constructed a non-tight gap in our chain. This seems to be a hopeless task.  

What is worse, the gaps constructed by Hausdorff are
indestructible, i.e.\ we cannot interpolate them even in $\omega_1$-preserving extensions of the universe  (see e.g.~\cite[Section 2]{Scheepers}). So, even in the forcing constructions we have to be quite careful. 

We will show two ways to avoid this problem. In the first construction, showing that splitting chains exist in the standard Cohen model, we will add generically elements of the chain, ensuring that all uncountable cuts which show up are tight and that their tightness will not be killed later on. In the second
construction we will have to change our method (as we want to have $\mathfrak{p}>\omega_1$ in the final model). This time we will keep all the gaps in the constructed chain destructible. In this way for every cut we will be able to split
(generically) sets spreading it. 

\subsection{Splitting chains after adding Cohen reals.}

Let $\mathbb{C}_{\kappa}$ be the forcing with $\kappa$ Cohen reals. We are going to prove the following.

\begin{thm}\label{cohen}
	If $\kappa$ is of uncountable cofinality, then in $V^{\mathbb{C}_{\kappa}}$ there is a splitting chain.
\end{thm}

This result was mentioned by Nyikos in
\cite{Nyikos}. He announced that its proof would appear in a
later paper which however never appeared. 


First, we will recall the standard forcing adding a set interpolating a given gap. 
\begin{defi}\label{P_G}
Let $\mathscr{G} = (\mathscr{L},\mathscr{R})$ be a pre-gap.
Let $\mathbb{P}_{\mathscr{G}}$ be defined in the following way: $p\in \mathbb{P}_\mathscr{G}$ if $p=(L_p, R_p,s_p)$, where
\begin{itemize}
	\item $s_p\in 2^{<\omega}$,
	\item $L_p$, $R_p$ are finite subsets of $\mathscr{L}$ and $\mathscr{R}$, respectively,
\item $L\setminus |s_p| \subseteq R$ for each $L\in L_p$ and $R\in R_p$.
\end{itemize} Denote $F_p = \{n\colon s_p(n)=1\}$. Now, $p\leq q$ if 
\begin{itemize}
	\item $s_q\subseteq s_p$,
	\item $L_q\subseteq L_p$, $R_q\subseteq R_p$,
	\item $\bigcup L_q \setminus |s_p| \subseteq F_p \setminus |s_p| \subseteq \bigcap R_q$.
\end{itemize}
\end{defi}

Say that a gap $\mathscr{G}$ is \emph{destructible} if there is a ccc (or just $\omega_1$-preserving) forcing notion which adds a set interpolating~$\mathscr{G}$. In fact, as the following fact shows, if there is such forcing interpolating the gap, then the above one would do the job, too.

\begin{fact}(\cite{Scheepers}) \label{char-destructibility} A gap $\mathscr{G}$ is destructible if and only if the forcing $\mathbb{P}_\mathscr{G}$ is ccc.
\end{fact}


\begin{lem}\label{coh} Assume that $\mathscr{G} = (\mathscr{L},\mathscr{R})$ a pre-gap. Then  $\mathbb{P}_\mathscr{G}$ adds generically a name $\dot{S}$ for a set interpolating $\mathscr{G}$. Moreover, $\dot{S}$ splits each set spreading $\mathscr{G}$ from the ground model.
	If $\mathscr{G}$ is countable, i.e. $|\mathscr{L}\cup\mathscr{R}|\leq \omega$, then $\mathbb{P}_\mathscr{G}$ is just  a Cohen forcing.
\end{lem}
\begin{proof} If $G$ is a $\mathbb{P}_\mathscr{G}$-generic and $\dot{S}$ is a name for $\bigcup_{p\in G} F_p$, then
	\[ \Vdash_{\mathbb{P}_\mathscr{G}}\mbox{``} \dot{S} \mbox{ interpolates }\mathscr{G}\mbox{"}.\]
	Now, assume that $A\in V$ spreads $\mathscr{G}$. Since \[ D_n = \{p\in \mathbb{P}_\mathscr{G}\colon \exists m>n \ m\in F_p \cap A \mbox{ and } \exists m'>n \ m'\in F_p \setminus A\} \] is dense for each $n$, 
	\[ \Vdash_{\mathbb{P}_\mathscr{G}} \mbox{``}\dot{S} \mbox{ splits } A\mbox{"}.\]
	Clearly, if $\mathscr{G}$ is countable, then $\mathbb{P}_\mathscr{G}$ is countable and atomless, so it is isomorphic to the Cohen forcing. 
\end{proof}

We will need one more fact. It is known that a Cohen forcing does not destroy towers (see e.g. \cite[Theorem 2.5]{Hirschorn00}). The argument can be easily modified to show the following. 

\begin{thm}\label{cohen-tight}
	Let $\mathscr{G}$ be a tight pre-gap. Adding any number of Cohen reals cannot add a subset spreading $\mathscr{G}$.
\end{thm}

\begin{proof} Denote by $\mathbb{C}$ the Cohen forcing. Assume that $\mathscr{G}= ( (L_\alpha)_{\alpha<\kappa},(R_\alpha)_{\alpha<\lambda} )$ is a tight pre-gap. We may assume that $\kappa$ and $\lambda$ are regular. We will show that
	\[ \Vdash_{\mathbb{C}} \mbox{``}\mathscr{G} \mbox{ is tight"}. \]
 Translating our task using the standard Cohen names for subsets of $\omega$ we have to show that there is no Borel function $f\colon 2^\omega \longrightarrow [\omega]^\omega$ such that $\{x\colon f(x)\cap L_\alpha =^* \varnothing\}$ and $\{x\colon
	f(x)\subseteq^* R_\alpha\}$ are comeager. 

	Indeed, suppose that such function $f$ exists and let $G\subseteq 2^\omega$ be a co-meager set such that $f|_{G}$ is continuous. Fix a countable base $\mathscr{U}$ of $2^\omega$. 
	Denote
	\[ L^n_\alpha = \{x\in G\colon f(x)\setminus n \cap L_\alpha = \varnothing\}.\]
	Since $f$ is continuous on $G$, each $L^n_\alpha$ is closed in $G$. By Baire theorem there is $n_\alpha$ and $U_\alpha \in \mathscr{U}$ such that $U_\alpha \cap G \subseteq L^{n_\alpha}_\alpha$.
Since $\mathscr{G}$ is tight, it has to be uncountable and so we may assume that $\lambda$ is uncountable. Hence, we
can find $n$, $U$ and $\Gamma\subseteq \lambda$ cofinal in $\lambda$ such that $n=n_\alpha$ and $U=U_\alpha$ for every $\alpha\in \Gamma$.  

We have to deal with two cases.	
\begin{itemize}
	\item \textbf{$\kappa$ is countable.} Let $R^n = \{x\in 2^\omega\colon f(x) \subseteq^* R_n\}$. Let $R=\bigcap_{n\in \lambda} R^n$. Then $R$ is comeager, since $R^n$ is comeager for each $n\in \lambda$. Pick $x\in R\cap U\cap G$ and let $A=f(x)$. 
Then $A\subseteq^* R_n$ for each $n$ and $L_\alpha \cap A =\varnothing$ for each
$\alpha\in \Gamma$. Thus, $L_\alpha \cap A =^* \varnothing$ for each $\alpha<\lambda$ and so $A$ spreads $\mathscr{G}$, a contradiction.

\item \textbf{$\kappa$ is uncountable.} Then, let \[ R^k_\alpha = \{x\in U\colon f(x)\setminus k \subseteq R_\alpha\}. \]
	Using the same argument as before we can find $n'>n$, $U'\subseteq U$ and $\Gamma'\subseteq \kappa$ cofinal in $\kappa$, such that 
	$U' \cap G \subseteq R^{n'}_\alpha$
	for $\alpha\in \Gamma'$. Let $x\in U'\cap G$ and let $A=f(x)$. Then $A\setminus n' \cap L_\alpha = \varnothing$ for each $\alpha\in \Gamma$ and $A\setminus n' \subseteq R_\alpha$ for each $\alpha\in \Gamma'$. Therefore, $A$ spreads $\mathscr{G}$, a contradiction.

\end{itemize}

	Since each set added by forcing with many Cohen reals can be added by a single Cohen real, we are done.
\end{proof}

\begin{proof}[Proof of Theorem \ref{cohen}] Let $V$ be a model with $\mathsf{CH}$. We will construct an iteration $(\mathbb{P}_\alpha)_{\alpha<\kappa}$ and a sequence $\dot{\mathscr{C}_\alpha}$ of names for $\subseteq^*$-chains such that for every
	$\alpha\leq \kappa$
	\begin{enumerate}
		\item\label{i} $\dot{\mathscr{C}_\alpha} \in V^{\mathbb{P}_\alpha}$,
		\item\label{ii} $\mathbb{P}_{\alpha+1} = \mathbb{P}_\alpha \star \mathbb{P}_{\dot{\mathscr{G}}}$, where $\dot{\mathscr{G}}$ is a
			name for a countable cut in $\dot{\mathscr{C}_\alpha}$, if there are countable cuts in $\dot{\mathscr{C}_\alpha}$
			or $\mathbb{P}_{\alpha+1}=\mathbb{P}_\alpha \star \mathbb{C}$ otherwise (where $\mathbb{C}$ is the standard Cohen forcing),
		\item\label{iii} $\dot{\mathscr{C}}_{\alpha+1}$ is the name for the chain $\dot{\mathscr{C}_\alpha} \cup \{\dot{S}\}$, where $\dot{S}$ is the name for a set added generically by $\mathbb{P}_{\alpha+1}$,
		\item\label{iv} if $\alpha$ is limit, then $\mathbb{P}_\alpha$ is the finite support iteration of $(\mathbb{P}_\xi)_{\xi<\alpha}$ and $\dot{\mathscr{C}_\alpha} = \bigcup_{\xi<\alpha} \dot{\mathscr{C}_\xi}$,
		\item\label{ti} $\Vdash_{\mathbb{P}_\alpha}$ ``each uncountable cut in $\dot{\mathscr{C}}_\alpha$ is tight``,
		\item\label{v} there are no countable cuts in $\dot{\mathscr{C}}_{\kappa}$.
	\end{enumerate}
	
	Let $\mathbb{P}_0$ be the trivial forcing and let $\mathscr{C}_0 =\{\varnothing, \omega\}$. 
	Then we can recursively define $\mathbb{P}_\alpha$ and $\dot{\mathscr{C}_\alpha}$
	satisfying (\ref{ii}) and (\ref{iii}), using the standard bookkeeping argument (and the fact that each countable cut in a
	$\subseteq^*$-chain can be interpolated and so there are at most $\mathfrak{c}$ countable cuts in a
	$\subseteq^*$-chain) to satisfy also (\ref{v}). 

	To show that the condition (\ref{ti}) will be satisfied we first prove the following claim.
	\medskip

	\textbf{Claim.} For each $\alpha<\kappa$
	\[ (\star_\alpha) \quad \Vdash_{\mathbb{P}_\alpha} \mbox{``each uncountable cut in }\dot{\mathscr{C}_\alpha}\mbox{ is tight".}\]

	We will prove it by induction on $\alpha$. First, notice that $\mathscr{C}_0$ does not contain any uncountable cut, so $(\star_0)$ is satisfied trivially. Suppose that there is $\alpha\leq \kappa$ such that $(\star_\xi)$ holds for each
	$\xi<\alpha$ and that $(\star_\alpha)$ does not hold. Since Cohen forcing cannot destroy tightness of a pre-gap (thanks to Theorem
	\ref{cohen-tight}) and it cannot add any new uncountable pre-gap (since
	each uncountable set of ordinals in the Cohen extension contains an uncountable subset from the ground model), $\alpha$
	has to be of uncountable cofinality. Let $\mathscr{G}$ be an uncountable cut in $\mathscr{C}_\alpha$ which is not tight and let $A\subseteq \omega$ spreads it. Since $\mathrm{cf}(\alpha)>\omega$, there is $\xi<\alpha$ such that $\dot{A}\in
	V^{\mathbb{P}_\xi}$. There is a cut $\mathscr{G}_0=(\mathscr{L},\mathscr{R})$ in $\mathscr{C}_\alpha$ such that $A$ spreads it. By the induction hypothesis	$\mathscr{G}_0$ has to be countable. Since $\mathscr{G}$ is uncountable, there is $\xi\leq \beta<\alpha$ such that  $\mathbb{P}_{\beta+1}=\mathbb{P}_\beta \star \mathbb{P}_\mathscr{H}$, where $\mathscr{H}$ is equivalent to $\mathscr{G}_0$ (in the sense
	that $\mathscr{H}$ and $\mathscr{G}_0$ have the same family of interpolating sets). But then using Lemma \ref{coh} we get that $\mathscr{C}_{\beta+1}$ splits $A$, a contradiction. The claim is proved.
	\medskip

   Of course the conjunction of (\ref{ti}) and (\ref{v}) implies that
   \[ \Vdash_{\mathbb{P}_{\kappa}} \mbox{ ``every cut in }\dot{\mathscr{C}}_{\kappa}\mbox{ is tight". }\]
   and so, by Proposition \ref{splitting-tight} we are done.
\end{proof}	

\subsection{Splitting chains with big $\mathfrak{p}$}

Theorem \ref{nyikos} and Theorem \ref{PFA} seem to suggest that the existence of splitting chains may be connected to the value of $\mathfrak{p}$. Indeed, if $\mathfrak{p}>\omega_1$, then a splitting chain could not have cuts which are $(\omega_1,\omega_1)$-gaps (as
they cannot be tight) and it is not obvious how to avoid them in the construction. However, we will show that splitting chains can exist even if $\mathfrak{p}>\omega_1$. The basic idea is to use iteration intertwining forcings destroying gaps from Definition \ref{P_G} with forcings adding pseudointersections to bases of filters on $\omega$.

\begin{defi}\label{Mathias-Prikry} Let $\mathscr{F}$ be a filter on $\omega$. The \emph{Mathias--Prikry forcing} $\mathbb{M}_\mathscr{F}$ is defined in the following way: $p\in \mathbb{M}_\mathscr{F}$ iff $p = (t_p, F_p)$, where $t_p\in
		2^{<\omega}$, $\mathrm{supp}(t_p) \cap F_p = \varnothing$ and $F_p\in \mathscr{F}$. $p\leq q$ if $t_q \subseteq t_p$, $F_p\subseteq F_q$ and $t_p(n)=0$ whenever $n\in (\mathrm{supp}(t_p)\setminus \mathrm{supp}(t_q))\setminus F_q$.
\end{defi}

Recall that $\mathbb{M}_\mathscr{F}$ \emph{diagonalizes} $\mathscr{F}$ (i.e. it adds a pseudo-intersection of $\mathscr{F}$; see e.g.
\cite{Mathias}).

Let $\kappa$ be a regular uncountable cardinal. For the rest of this section fix two subsets $\Gamma$, $\Lambda\subseteq \kappa$ which form a partition of $\kappa$ into cofinal subsets. In our construction, at steps from $\Gamma$ we will add sets interpolating cuts, and at steps from $\Lambda$ we will diagonalize
filters. Namely, we start with a model with $\mathsf{GCH}$ and perform a finite support iteration $(\mathbb{Q}_\alpha, \mathbb{P}_\alpha)_{\alpha<\kappa}$, where $\mathbb{Q}_0$ is the trivial forcing, $\mathbb{Q}_{\alpha+1} = \mathbb{Q}_\alpha \star \mathbb{P}_\alpha$ for every $\alpha<\kappa$.
Moreover, for each $\alpha<\kappa$ the forcing $\mathbb{P}_\alpha$ is either trivial or
\begin{itemize}
	\item for $\alpha \in \Gamma$ it is of the form $\mathbb{M}_{\dot{\mathscr{F}}}$, where $\dot{\mathscr{F}}$ is a $\mathbb{Q}_\alpha$-name for a filter generated by less than $\kappa$ sets. 
	\item for $\alpha \in \Lambda$ it is of the form $\mathbb{P}_{\dot {\mathscr{G}}}$, where $\dot{\mathscr{G}}$ is a $\mathbb{Q}_\alpha$-name for a cut in the chain $\{\dot{S}_\beta\colon \beta\in \Lambda\cap
		\alpha\} \cup \{\varnothing, \omega\}$, where $\dot{S}_\beta$ is the $\mathbb{Q}_\beta$-name for a subset of $\omega$ added generically by $\mathbb{Q}_\beta$.
\end{itemize}
Let $\mathbb{Q}_\kappa$ be the limit of the iteration.

Note that in this definition we \emph{a priori} assume that $\mathbb{Q}_\alpha$ forces $\{\dot{S}_\beta\colon \beta\in \Lambda\cap
\alpha\}$ to be a chain. That this is the case can be shown by induction using Lemma \ref{coh}: if $\alpha\in \Lambda\cap \kappa$, $\mathscr{S} = \{\dot{S}_\beta\colon \beta\in \Lambda\cap \alpha\}$ forms a chain, and $\dot{\mathscr{G}}$ is a
$\mathbb{Q}_\alpha$-name
for a cut in $\mathscr{S}$, then $\mathbb{P}_{\alpha} = \mathbb{P}_\mathscr{G}$ adds a set $\dot{S}_\alpha$ interpolating the cut and thus $\mathbb{Q}_\alpha \star \mathbb{P}_\mathscr{G}$ forces $\mathscr{S}\cup \{\dot{S}_\alpha\}$ to be a chain.

In what follows we will make a cosmetic change in the definition of $\mathbb{P}_\mathscr{G}$. At step $\alpha$ of the iteration, as $\mathscr{G}\subseteq \{\dot{S}_\beta\colon \beta<\alpha\} \cup \{\varnothing,\omega\}$ and thus elements of $\mathscr{G}$ are naturally indexed by elements of $\alpha$, the sets
$L_p$ and $R_p$, for $p\in \mathbb{P}_\alpha$, will be subsets of $[\alpha]^{<\omega}$ (instead of $[\mathscr{G}]^{<\omega}$). (To avoid problems with $\varnothing$ and $\omega$, we may assume that $0,1\in \Gamma$ and $S_0=\varnothing$, $S_1=\omega$.)

We will prove inductively that regardless of the choice of the names for the filters and cuts, the forcing $\mathbb{Q}_\alpha$ is ccc for every $\alpha\leq \kappa$. We will use arguments from \cite{Laver}. Although Laver's construction serves for different purposes and concerns a different structure, in fact we follow the path of his proof quite strictly. 

Notice that usually to prove that a finite support iteration is ccc, one uses the preservation theorem for finite support iterations of ccc forcings. This time, the fact that the iterands are ccc will be rather a conclusion of the fact that the whole
iteration is ccc. Further conclusion is that all the cuts in the generically added chain which form gaps are destructible
(see Fact \ref{char-destructibility}).

\begin{thm}\label{ccc} $\mathbb{Q}_\alpha$ is ccc for every $\alpha\leq \kappa$.
\end{thm}

To prove the theorem we will need several lemmas.

\begin{defi}\label{uniform-conditions}
		Let $\mathbb{R}_\alpha$ be the set of conditions $p$ in $\mathbb{Q}_\alpha$ satisfying the following properties: 
	\begin{enumerate}
		\item  $p_{|\beta}$ decides $(L_{p(\beta)}, R_{p(\beta)}, s_{p(\beta)})$ for $\beta \in \Lambda \cap \mathrm{supp}(p)$ and $p_{|\beta}$ decides $t_{p(\beta)}$ for $\beta \in \Gamma\cap \mathrm{supp}(p)$,
		\item for each $\beta, \gamma \in \Lambda \cap \mathrm{supp}(p)$ we have $\max(s_{p(\gamma)}) = \max(s_{p(\beta)})$ (in such case denote this maximum by $\ell_p$),
		\item\label{closed-support} for each $\beta \in \Lambda \cap \mathrm{supp}(p)$ we have $L_{p(\beta)}\cup R_{p(\beta)} \subseteq  \mathrm{supp}(p)$.
	\end{enumerate}
\end{defi}

The following simple lemma says that we may work only with conditions in $\mathbb{R}_\alpha$. 

\begin{lem} For each $\alpha$ the set $\mathbb{R}_\alpha$ is dense in $\mathbb{Q}_\alpha$.
\end{lem}

\begin{proof}
	We will prove it inductively on $\alpha$. Since the iteration is of finite support, the limits steps are obvious. Let $\alpha<\kappa$ and consider $p\in \mathbb{Q}_{\alpha+1}$. Denote $q=p_{|\alpha}$. Find $q'\leq  q$ so that 
	\begin{itemize}
		\item $q'$ decides $p(\alpha)$,
		\item $L_{q'(\alpha)} \cup R_{q'(\alpha)} \subseteq \mathrm{supp}(q')$,
		\item $\max(s_{q'(\beta)})>\max(s_{p(\alpha)})$ for some $\beta \in \mathrm{supp}(q')$.
	\end{itemize}
	Use the inductive hypothesis to find $r\leq q'$, $r\in \mathbb{R}_\alpha$. Now, notice that there is $s$ such that $\max(s)=\ell_p$ and $r^\frown(L_{p(\alpha)},R_{p(\alpha)},s) \leq
r^\frown(L_{p(\alpha)},R_{p(\alpha)},s_{p(\alpha)})$. Indeed, the existence of such $s$ follows from the fact that $\ell_p \geq \max(s_\alpha)$ and $L_{p(\alpha)}\cup R_{p(\alpha)} \subseteq \mathrm{supp}(r)$ (and so $s_\alpha$ can be appropriately extended). Clearly,
$r^\frown(L_{p(\alpha)}, R_{p(\alpha)}, s)$ is in $\mathbb{R}_\alpha$.
\end{proof}

\begin{lem}\label{c10} Let $p\in \mathbb{R}_\alpha$. If $\beta, \gamma \in \mathrm{supp}(p) \cap \Lambda$, $\beta<\gamma$ and $p \Vdash \dot{S}_\beta \subseteq^* \dot{S}_\gamma$ ($p \Vdash \dot{S}_\gamma \subseteq^* \dot{S}_\beta$), then there is $r$ such that
	\begin{itemize}
		\item $\mathrm{supp}(r)=\mathrm{supp}(p)$ and $s_{r(\delta)}=s_{p(\delta)}$ for each $\delta \in \mathrm{supp}(p)\cap \Lambda$, $t_{r(\delta)}=t_{p(\delta)}$ for $\delta\in \mathrm{supp}(p)\cap \Gamma$,
		\item $\beta\in L_{r(\gamma)}$ ($\beta\in R_{r(\gamma)}$).
	\end{itemize}
\end{lem}

	\begin{proof} Assume that $p\Vdash \dot{S}_\beta \subseteq^* \dot{S}_\gamma$ (the other case is clearly symmetric). We will prove the lemma by induction on $\alpha$. The limit step is obvious so assume that $\alpha<\kappa$ and consider
		$\alpha+1$'s step. We may assume that $\gamma=\alpha$. 

		Notice that $p_{|\alpha}\Vdash \dot{S}_\beta \subseteq^* \dot{S}_\alpha$. Hence, for each $\delta\in R_{p(\alpha)}$ we have $p_{|\alpha} \Vdash \dot{S}_\beta \subseteq^* \dot{S}_\delta$. By inductive hypothesis used $R_{p(\alpha)}$ many times
		we may find $q \leq  p_{|\alpha}$ as in the lemma, such that $\beta \in L_{q(\delta)}$ or $\delta\in L_{q(\beta)}$ for every $\delta\in R_{p(\alpha)}$. Finally, let
		\[ r = q^\frown (L_{p(\alpha)}\cup \{\beta\}, R_{p(\alpha)}, s_{p(\alpha)}).\vspace{-20pt} \]
\end{proof}

\begin{lem}\label{c11} Let $p,q\in \mathbb{R}_\alpha$ be such that $\ell_p = \ell_q$ and $s_{p(\beta)} = s_{q(\beta)}$ for each $\beta \in \Lambda \cap \mathrm{supp}(p)\cap \mathrm{supp}(q)$ and $t_{p(\beta)} = t_{q(\beta)}$ for $\beta\in \Gamma \cap \mathrm{supp}(p)\cap \mathrm{supp}(q)$. Then there is $r\leq p,q$.
\end{lem}
\begin{proof}
	As before, we will prove it inductively on $\alpha$. In fact, to make induction work, we will prove a stronger statement: we show that under the above conditions there is $r\leq p,q$ such that $r\in \mathbb{R}_\alpha$ and $\ell_r = \ell_p$.

Again, the limit step is clear, so let $\alpha<\kappa$ and consider $p,q\in \mathbb{Q}_{\alpha+1}$. Define $r'$ in the following way:
	\begin{itemize}
		\item if $\mathrm{supp}(p_{|\alpha})$ is empty, then let $r'=q_{|\alpha}$, 
		\item if $\mathrm{supp}(q_{|\alpha})$ is empty, then let $r'=p_{|\alpha}$, 
		\item if both of the supports are non-empty, then let $r'$ be given by the inductive hypothesis for $p_{|\alpha}$ and $q_{|\alpha}$.
	\end{itemize}
	We may assume that $\alpha \in \mathrm{supp}(p)$, otherwise $r = r'^\frown q(\alpha)$ will be as desired. Similarly, we assume that $\alpha\in \mathrm{supp}(q)$. 

	Suppose first that $\alpha\in \Gamma$. Then it is enough to take 
	\[ r = r'^\frown (t_{p(\alpha)}, F_{p(\alpha)}\cap F_{q(\alpha)}). \]

 If $\alpha\in \Lambda$, then notice that
 \[ r'\Vdash \dot{S}_\beta \subseteq^* \dot{S}_\gamma\mbox{ for each }\beta\in L_{p(\alpha)}\cup L_{q(\alpha)}, \gamma \in R_{p(\alpha)}\cup R_{q(\alpha)}. \]
 Indeed, if $\beta\in L_{p(\alpha)}$ ($\beta\in L_{q(\alpha)}$) and $\gamma\in R_{p(\alpha)}$ ($\gamma \in R_{q(\alpha)}$), then $\beta,\gamma \in \mathrm{supp}(p_{|\alpha})$ ($\beta,\gamma\in \mathrm{supp}(q_{|\alpha})$) and so $p_{|\alpha}$ ($q_{|\alpha}$) forces that $\dot{S}_\beta \subseteq^* \dot{S}_\gamma$.

 Notice that extending $r'$ in the obvious way on $\alpha$ would give us a condition stronger than $p$ and $q$ but not necessarily in $\mathbb{R}_\alpha$. To fulfil condition (\ref{closed-support}) of Definition \ref{uniform-conditions} we have to  apply subsequently Lemma \ref{c10} to find $r''\leq r'$ such that 
 \begin{itemize}
	 \item for every $\beta\in L_{p(\alpha)}\cup L_{q(\alpha)}$ and every $\gamma \in R_{p(\alpha)}\cup R_{q(\alpha)}$ we have $\beta\in L_{r(\gamma)}$ or $\gamma\in R_{r(\beta)}$.
	 \item $\mathrm{supp}(r'') = \mathrm{supp}(r')$ and $s_{r''(\beta)} = s_{r'(\beta)}$ for each $\beta \in \mathrm{supp}(r')$.
 \end{itemize}
 Take $r = r''^\frown (L_{p(\alpha)}\cup L_{q(\alpha)}, R_{p(\alpha)}\cup R_{q(\alpha)}, s_{p(\alpha)})$.
\end{proof}

Now, we are ready to prove Theorem \ref{ccc}. 

\begin{proof}(of Theorem \ref{ccc}) Suppose that $\mathscr{P}$ is an uncountable subset of $\mathbb{Q}_\alpha$. We may assume that $\mathscr{P}\subseteq \mathbb{R}_\alpha$, that $\ell_p = \ell_q$ for each $p,q\in \mathscr{P}$ and, finally, that the
	supports of elements in $\mathscr{P}$ form a $\Delta$-system with a root $R$. Again, shrinking $\mathscr{P}$ if needed, we may assume that for each $p,q\in \mathscr{P}$ we have $s_{p(\alpha)} = s_{q(\alpha)}$ for $\alpha \in R\cap \Lambda$ and
	$t_{p(\alpha)} = t_{q(\alpha)}$ for $\alpha \in R\cap \Gamma$. Now, use Lemma \ref{c11}. 
\end{proof}


\begin{thm} The existence of a splitting chain is consistent with $\mathfrak{p}=\kappa$ for each regular uncountable $\kappa$.
\end{thm}
\begin{proof}
	Assume that $\dot{A}$ is a $\mathbb{Q}_\alpha$-name for an infinite subset of $\omega$. If it is not split by the chain $\{\dot{S}_\beta\colon \beta<\alpha\}$, then let $L = \{\beta\colon \dot{S}_\beta \cap \dot{R} =^* \varnothing\}$ and $R =
	\{\beta\colon \dot{R}\subseteq^* \dot{S}_\beta\}$. Let $\mathscr{G} = (\{\dot{S}_\beta\colon \beta\in L\}, \{\dot{S}_\beta\colon \beta\in R\})$. By Theorem \ref{ccc} the forcing $\mathbb{P}_\mathscr{G}$ is ccc and it adds a generic set
	$\dot{S}\subseteq \omega$ interpolating $\mathscr{G}$. Moreover, by Lemma \ref{coh}, $\dot{S}$ splits $\dot{A}$.

	Using this remark, we can apply the standard bookkeeping machinery over all the $\mathbb{Q}_\alpha$-names, $\alpha\in \Lambda$, for infinite subsets of $\omega$ to ensure that all of them will appear as a subset defining a cut as above which will be used to define
	$\mathbb{P}_\alpha$ for some $\alpha\in \Lambda$ (if at step $\alpha$ the subset is already split by the previously constructed chain, then let $\mathbb{P}_\alpha$ be the trivial forcing). 
	
Simultaneously, we apply the bookkeeping over all the $\mathbb{Q}_{\alpha}$-names, $\alpha\in \Gamma$, for filters generated by less than $\kappa$ sets to ensure that all of them will be diagonalized in the process of iteration.
\end{proof}

\bibliographystyle{alpha}
\bibliography{snall}

\newcommand{\etalchar}[1]{$^{#1}$}
\begin{thebibliography}{CSCKY03}

\bibitem[ACSC{\etalchar{+}}13]{avilesAiM}
Antonio Avil\'es, F\'elix Cabello~S\'anchez, Castillo, Jes\'us~M. F., Manuel
  Gonz\'alez, and Yolanda Moreno.
\newblock On separably injective {B}anach spaces.
\newblock {\em Adv. Math.}, 234:192--216, 2013.

\bibitem[ACSC{\etalchar{+}}16]{aviles}
Antonio Avil\'es, F\'elix Cabello~S\'anchez, Jes\'us M.~F. Castillo, Manuel
  Gonz\'alez, and Yolanda Moreno.
\newblock {\em Separably injective {B}anach spaces}, volume 2132 of {\em
  Lecture Notes in Mathematics}.
\newblock Springer, [Cham], 2016.

\bibitem[ACSC{\etalchar{+}}17]{corr}
Antonio Avil\'es, F\'elix Cabello~S\'anchez, Jes\'us M.~F. Castillo, Manuel
  Gonz\'alez, and Yolanda Moreno.
\newblock Corrigendum to ``{O}n separably injective {B}anach spaces'' [{A}dv.
  {M}ath. 234 (2013) 192--216] [ {MR}3003929].
\newblock {\em Adv. Math.}, 318:737--747, 2017.

\bibitem[AK16]{Albiac-Kalton}
Fernando Albiac and Nigel~J. Kalton.
\newblock {\em Topics in {B}anach space theory}, volume 233 of {\em Graduate
  Texts in Mathematics}.
\newblock Springer, [Cham], second edition, 2016.
\newblock With a foreword by Gilles Godefory.

\bibitem[Ami64]{amir}
Dan Amir.
\newblock Projections onto continuous function spaces.
\newblock {\em Proc. Amer. Math. Soc.}, 15:396--402, 1964.

\bibitem[Ark92]{Arhan}
Alexander~V. Arkhangelski\u\i.
\newblock {\em Topological function spaces}, volume~78 of {\em Mathematics and
  its Applications (Soviet Series)}.
\newblock Kluwer Academic Publishers Group, Dordrecht, 1992.
\newblock Translated from the Russian by R. A. M. Hoksbergen.

\bibitem[AT11]{multigaps}
Antonio Avil\'{e}s and Stevo Todorcevic.
\newblock Multiple gaps.
\newblock {\em Fund. Math.}, 213(1):15--42, 2011.

\bibitem[Ban77]{bankston}
Paul Bankston.
\newblock Ultraproducts in topology.
\newblock {\em General Topology and Appl.}, 7(3):283--308, 1977.

\bibitem[BL00]{Benyamini}
Yoav Benyamini and Joram Lindenstrauss.
\newblock {\em Geometric nonlinear functional analysis. {V}ol. 1}, volume~48 of
  {\em American Mathematical Society Colloquium Publications}.
\newblock American Mathematical Society, Providence, RI, 2000.

\bibitem[BP58]{Bessaga58}
Czes{\l}aw Bessaga and Aleksander Pe{\l}czy{\'n}ski.
\newblock On bases and unconditional convergence of series in {B}anach spaces.
\newblock {\em Studia Math.}, 17:151--164, 1958.

\bibitem[CG97]{castillo}
Jes\'{u}s M.~F. Castillo and Manuel Gonz\'{a}lez.
\newblock {\em Three-space problems in {B}anach space theory}, volume 1667 of
  {\em Lecture Notes in Mathematics}.
\newblock Springer-Verlag, Berlin, 1997.

\bibitem[CSCKY03]{Felix03}
F\'{e}lix Cabello~S\'{a}nchez, J\'{e}sus M.~F. Castillo, Nigel~J. Kalton, and
  David~T. Yost.
\newblock Twisted sums with {$C(K)$} spaces.
\newblock {\em Trans. Amer. Math. Soc.}, 355(11):4523--4541, 2003.

\bibitem[Dit73]{Ditor}
Seymour~Z. Ditor.
\newblock Averaging operators in {$C(S)$} and lower semicontinuous sections of
  continuous maps.
\newblock {\em Trans. Amer. Math. Soc.}, 175:195--208, 1973.

\bibitem[Ev78]{eficer}
V.~A. Efimov and G.~I. \v{C}ertanov.
\newblock Subspaces of {$\Sigma $}-products of intervals.
\newblock {\em Comment. Math. Univ. Carolin.}, 19(3):569--593, 1978.

\bibitem[FS65]{f-s}
Ciprian Foia\c{s} and Ivan Singer.
\newblock On bases in {$C([O,\,1])$} and {$L^{1}\,([0,\,1])$}.
\newblock {\em Rev. Roumaine Math. Pures Appl.}, 10:931--960, 1965.

\bibitem[HI02]{henson}
C.~Ward Henson and Jos\'{e} Iovino.
\newblock Ultraproducts in analysis.
\newblock In {\em Analysis and logic ({M}ons, 1997)}, volume 262 of {\em London
  Math. Soc. Lecture Note Ser.}, pages 1--110. Cambridge Univ. Press,
  Cambridge, 2002.

\bibitem[Hir00]{Hirschorn00}
James Hirschorn.
\newblock Towers of {B}orel functions.
\newblock {\em Proc. Amer. Math. Soc.}, 128(2):599--604, 2000.

\bibitem[Lav79]{Laver}
Richard Laver.
\newblock Linear orders in {$(\omega )^{\omega }$} under eventual dominance.
\newblock In {\em Logic {C}olloquium '78 ({M}ons, 1978)}, volume~97 of {\em
  Stud. Logic Foundations Math.}, pages 299--302. North-Holland, Amsterdam-New
  York, 1979.

\bibitem[Mat77]{Mathias}
Adrian R.~D. Mathias.
\newblock Happy families.
\newblock {\em Ann. Math. Logic}, 12(1):59--111, 1977.

\bibitem[MP18]{Plebanek18}
Witold Marciszewski and Grzegorz Plebanek.
\newblock Extension operators and twisted sums of {$c_0$} and {$C(K)$} spaces.
\newblock {\em J. Funct. Anal.}, 274(5):1491--1529, 2018.

\bibitem[Nie17]{Niewczas}
Maciej Niewczas.
\newblock Tunele w przestrzeniach topologicznych (in {P}olish).
\newblock {\em Bachelor thesis, Wroc\l aw}, 2017.

\bibitem[NV83]{Nyikos83}
Peter~J. Nyikos and Jerry~E. Vaughan.
\newblock On first countable, countably compact spaces. {I}. {$(\omega
  _{1},\,\omega ^{\ast} _{1})$}-gaps.
\newblock {\em Trans. Amer. Math. Soc.}, 279(2):463--469, 1983.

\bibitem[Nyi88]{Nyikos}
Peter~J. Nyikos.
\newblock The complete tunnel axiom.
\newblock {\em Topology Appl.}, 29(1):1--18, 1988.

\bibitem[Ros74]{Rosenthal74}
Haskell~P. Rosenthal.
\newblock A characterization of {B}anach spaces containing {$l^{1}$}.
\newblock {\em Proc. Nat. Acad. Sci. U.S.A.}, 71:2411--2413, 1974.

\bibitem[Sch93]{Scheepers}
Marion Scheepers.
\newblock Gaps in {$\omega^\omega$}.
\newblock In {\em Set theory of the reals ({R}amat {G}an, 1991)}, volume~6 of
  {\em Israel Math. Conf. Proc.}, pages 439--561. Bar-Ilan Univ., Ramat Gan,
  1993.

\bibitem[Ste78]{stern}
Jacques Stern.
\newblock Ultrapowers and local properties of {B}anach spaces.
\newblock {\em Trans. Amer. Math. Soc.}, 240:231--252, 1978.

\bibitem[Tod89]{Todorcevic89}
Stevo Todor\v{c}evi\'c.
\newblock {\em Partition problems in topology}, volume~84 of {\em Contemporary
  Mathematics}.
\newblock American Mathematical Society, Providence, RI, 1989.

\end{thebibliography}

\end{document}